\documentclass[12pt, reqno]{amsart}
\usepackage{amsfonts,amsmath,amssymb,amsthm,epsfig,cite,graphicx,hyperref,
	color, esint,fancyhdr, enumerate, latexsym,amsrefs, makecell}
\usepackage{color}
\usepackage[mathscr]{eucal}
\usepackage{comment}

\usepackage[nameinlink]{cleveref}

\numberwithin{equation}{section}

\theoremstyle{definition}
\newtheorem{definition}{Definition}[section]
\newtheorem{example}[definition]{Example}

\theoremstyle{remark}
\newtheorem{remark}[definition]{Remark}
\theoremstyle{plain}
\newtheorem{theorem}[definition]{Theorem}
\newtheorem{lemma}[definition]{Lemma}
\newtheorem{proposition}[definition]{Proposition}

\newtheorem{result}[definition]{Result}

\newtheorem{corollary}[definition]{Corollary}

\newcommand{\eps}{\varepsilon}

\newcommand{\Om}{\Omega}

\newcommand{\st}{\subset}
\newcommand{\St}{\Subset}
\newcommand{\steq}{\subset}

\newcommand{\OOm}{\overline{\Omega}}
\newcommand{\mco}{\mathcal{O}}
\newcommand{\D}{\mathbb{D}}

\newcommand{\disc}{\mathbb{D}}

\newcommand{\smoo}{\mathcal{C}}
\newcommand{\hol}{\mathcal{O}}

\newcommand{\rl}{{\sf Re}}

\newcommand{\impl}{\Longrightarrow}

\newcommand\wtil[1]{\widetilde{#1}}
\newcommand\ba[1]{\overline{#1}}

\newcommand{\CC}{\mathbb{C}^2}
\newcommand{\cplx}{\mathbb{C}}

\newcommand{\re}{\mathbb{R}}
\newcommand{\cn}{\mathbb{C}^n}

\begin{document}
	\title[A study of spirallike domains]{A study of spirallike domains: polynomial convexity, Loewner chains and dense holomorphic curves}
	\author{Sanjoy Chatterjee and Sushil Gorai}
	\address{Department of Mathematics and Statistics, Indian Institute of Science Education and Research Kolkata,
		Mohanpur -- 741 246}
	\email{sc16ip002@iiserkol.ac.in}
	
	\address{Department of Mathematics and Statistics, Indian Institute of Science Education and Research Kolkata,
		Mohanpur -- 741 246}
	\email{sushil.gorai@iiserkol.ac.in, sushil.gorai@gmail.com}
	\thanks{Sanjoy Chatterjee is supported by CSIR fellowship (File No-09/921(0283)/2019-EMR-I). Sushil Gorai is partially supported by a Core Research Grant (CRG/2022/003560) of SERB, Govt. of India}

	\date{\today}


	\begin{abstract}
		In this paper, we  prove that the closure of a bounded pseudoconvex domain, which is spirallike with respect to a globally asymptotic stable holomorphic vector field, is polynomially convex. We also provide a necessary and sufficient condition, in terms of polynomial convexity, on
a univalent function defined on a strongly convex domain for embedding it into a filtering Loewner chain. Next, we provide an application of our first result.
We show that for any bounded pseudoconvex strictly spirallike domain $\Omega$ in $\cplx^n$ and given any connected complex manifold $Y$, there exists a holomorphic map from the unit disc to the space of all holomorphic maps from $\Omega$ to $Y$. 
This also yields us the existence of $\hol(\Omega, Y)$-universal map for any generalized translation on $\Omega$, which, in turn, is connected to the hypercyclicity of certain composition operators on the space of manifold valued holomorphic maps.
	\end{abstract}	 

	\maketitle

	\section{Introduction and statements of the results}
	
 The domains we study in this paper are pseudoconvex domains that are spirallike with respect to certain holomorphic vector fields.
	Recall that
		a holomorphic vector field $V$ on a domain  $\Om \subset \mathbb{C}^{n}$ is a real vector field on $\Om$  such that  $$V(z)=\sum_{i=1}^{n} a_{i}(z) \frac{\partial }{\partial  x_{i}}+b_{i}(z) \frac{\partial }{\partial  y_{i}},$$ where $(a_j(z)+ib_{j}(z))$ is holomorphic function on $\Om$  for all $j \in \{1,2, \cdots ,n\}$.  We denote the set of all holomorphic vector fields on $\cn$ by $\mathfrak{X}_{\mathcal{O}}(\cn)$.
	 Any  holomorphic map $F\colon \cn \to \cn$ can be viewed as a holomorphic vector field on $\cn$. We will use matrices with complex entries while talking about linear vector fields on $\cn$.
 We will need the  notion of spirallike domain with respect to a holomorphic vector field from \cite{CG} to move further into our discussion.
 
	\begin{definition}\label{D:def-spirallike}
		Let $\Omega$ and $\widetilde{\Om}$ be domains in $\mathbb{C}^{n}$, such that $0 \in \Om \subset \OOm \st \widetilde{\Om} \subseteq \cn$. Suppose that $\Phi$ be a holomorphic vector field on $\wtil{\Omega}$  such that $\Phi(0)=0$. Then $\Omega$ is  said to be spirallike with respect to $\Phi$, if for any $z\in \Omega$, the initial value problem 
		\begin{align*}
			\frac{dX}{dt}&=\Phi(X(t))\\
			X(0)&=z,
		\end{align*}	
		has a solution defined for all $t \geq 0$ with $X(t,z) \in \Om$ for all $t>0$ and  $X(t) \to 0$ as $t \to \infty$. We say that $\Om$ is {\em strictly spirallike} with respect to holomorphic vector field $\Phi$ if $X(t,z) \in \Om$ for all $t>0$ and for all $z \in \overline{\Om}$.
	\end{definition}	

 \noindent Next, we briefly mention some notions of stability of the equilibrium point of a system of differential equations (see \cite{perko} for details). The notions of stability will play a vital role in our study.   
	Let $ E \subset \mathbb{R}^{n}$ be an open set containing the origin. Suppose that   $f\colon E \to \mathbb{R}^{n}$ is a continuously differentiable mapping such that $f(0)=0$. Consider the system of differential equation
	\begin{align}\label{eq-ddeq1}
		\frac{dX(t)}{dt}&  = f(X(t)) ,~~
		X(0)  =x_{0}.
	\end{align}
	Assume that  the solution of the system \eqref{eq-ddeq1} 
	exists for every $t \geq 0$ and $\forall x_{0} \in E$. Then:
	\begin{itemize}
		\item The origin is said to be  a {\em stable equilibrium point} of the system \eqref{eq-ddeq1}, if  for every $\epsilon>0$ there exist $\delta >0$ such that $X(t,x_{0}) \in B(0, \epsilon)$ for every $x_{0} \in B(0, \delta)$. 
		\smallskip
		
		\item The origin is said to be {\em globally asymptotically stable equilibrium point} with $E=\re^n$ if the origin is stable and $\lim_{t\to \infty} X(t,x_{0}) =0$ for all $x_{0} \in \re^n$.
  \end{itemize}
  A vector field $V \in \mathfrak{X}_{\mathcal{O}}(\cn)$  is said to be {\em globally asymptotically stable vector field} if the origin is the  globally asymptotically stable equilibrium point of $V$. In this paper, we always consider  globally asymptotic stable vector field whose equilibrium point is the origin.
		\smallskip
	
In this paper we study the polynomial convexity of the closure of pseudoconvex strictly spirallike domains. We also study the embedding a univalent function into a filtering Loewner chain through polynomial convexity property along the lines of Hamada \cite{hamada2020}. We also use polynomial convexity of certain domains to study the dense holomorphic curves in the space of all holomorphic maps. We now describe each of these separately in the following subsections.

\subsection{Polynomial convexity} For a compact subset $K\subset \cn $ the Polynomially convex hull of $K$, denoted by $\widehat{K}$, is defined by  
$$
\widehat{K}:=\{z \in \cn:|p(z)| \leq \sup_{w \in K}|p(w)|, \forall p \in \mathbb{C}[z_{1},z_{2}, \cdots ,z_{n}]\}.
$$
We say that $K$ is polynomially convex if $\widehat{K}=K$. Polynomial convexity is the main ingredient in the study of uniform approximation by polynomials. In $\mathbb{C}$, a compact subset $K$ is polynomially convex if and only if $\mathbb{C}\setminus K$ is connected. In general, for $n >1$, it is difficult to determine whether a compact subset in $\cn$ is polynomially convex or not.  It is known that any compact convex subset of $\cplx^n$ is polynomially convex. In particular, the closure of any bounded convex domain is polynomially convex.  However, in  \cite[Example~2.7]{Izzo93}, it was shown that the closure of a bounded strongly pseudoconvex domain with a smooth boundary may not be polynomially convex. In \cite{joita2007}, Joi\c{t}a gave an example of a  strongly  pseudoconvex domain in $\cplx^n$ with real analytic boundary whose closure is not polynomially convex. The doamin in the example of Joi\c{t}a\cite{joita2007} is a is also a Runge domain.
This raises a natural question: {\em For which classes of pseudoconvex domains the closure is polynomially convex?} Hamada\cite{hamada2020} proved that bounded pseudoconvex domains, which are strictly spirallike with respect to certain linear vector field, have  polynomially convex closure. 
\begin{result}[Hamada]\label{T:PClinear}
		Let $A\in M_n(\cplx)$ such that $\inf_{||z||=1}\rl\langle Az, z\rangle>0$. Let $\Omega  \subset \cn$
		be a bounded pseudoconvex domain containing the origin such that
$e^{-tA}w \in \Om$ for all $t > 0$ and for all $w \in \ba{\Omega}$. Then, $\overline{\Omega}$ is polynomially convex.
	\end{result}

In this article, we are able to provide a  generalization of \Cref{T:PClinear}, which enlarges the  class of bounded pseudoconvex domains that have polynomially convex closure. We need the following definition for the demonstration of our results.
	The first result of this paper states that  the conclusion of \Cref{T:PClinear} is also true if the domain $D$ is spirallike with respect to any asymptotic stable holomorphic vector field. More precisely, we present:
	\begin{theorem}\label{T:Polynostrict}
		Let $V \in \mathfrak{X}_{\mathcal{O}}(\cn)(n \geq 2)$ be a complete asymptotic stable vector field. Let $ D \subseteq \cn$ be a bounded  pseudoconvex domain containing the origin.  If  there exists $\psi \in Aut(\cn)$ such that  $\psi(D)$ is strictly spirallike domain with respect to $V$  then $\overline{D}$ is polynomially convex.
	\end{theorem} 

\begin{remark}
    It is proved in \cite{CG} (in \cite{Hamada15} for linear case) that any spirallike domain with respect to an asymptotically stable vector field is Runge. But, in view of Joi\c{t}a\cite{joita2007}, this does not imply the closure of the domain, in case the domain is bounded, is polynomially convex.
    Hamada also provided an example in \cite{hamada2020} showing that strictly spirallike is crucial, even in case the vector field is linear. This suggests that the strictly spirallike assumption in 
    \Cref{T:Polynostrict} is a natural assumption.
\end{remark}

\noindent The following corollary gives a condition of polynomial convex closure in terms of the defining function of the domain. For a holomorphic vector field $V(z)=\sum_{i=1}^{n} a_{i}(z) \frac{\partial }{\partial  x_{i}}+b_{i}(z) \frac{\partial }{\partial  y_{i}},$ we define $\wtil{V}(f)(z):=\sum_{j=1}^{n}(a_{j}(z)+ib_{j}(z))\frac{\partial f}{\partial z_{j}}$, where $f \in \smoo(\Om)$.
\begin{corollary}\label{T:polystrict1}
Let $D  \subset \cn$ be a bounded pseudoconvex  domain  with $\smoo^{\alpha}$ boundary, for some $\alpha \geq 1$ that contains the origin. Let $V \in \mathfrak{X}_{\mathcal{O}}(\cn)$ be a complete globally asymptotic stable vector field. Suppose that $\theta\colon \mathbb{R} \times \cn \to \cn$ be the flow of the vector field $V$. Assume that  $U \st \cn$ be an open subset such that  $\{\theta(t,z)|t \geq 0, z \in \overline{D}\} \subset  U$ and $r\colon U  \to \mathbb{R} $ be a defining function of $D$.  If $\rl(\wtil{V}(r))<0$ on $U$, then $\overline{D}$ is  polynomially convex.
\end{corollary}

\subsection{Loewner chains} 
The issue of embedding univalent functions within Loewner chains is the subject of our next discussion. Loewner\cites{Loe} invented a technique, now referred to as Loewner chains, for embedding univalent functions within particular families of univalent functions. The Loewner theory on the Kobayashi hyperbolic complex manifold was studied in \cites{BCM22009}. Poreda studied the Loewner chain on the polydisc in his papers \cites{Poreda87a, Poreda87b}, and \cite{Sch18} (see \cite{HamadaBall2021}, \cite{Hamada2021} and the references therein for an overview of recent results about the embedding of the univalent maps into the Loewner chain).  The embedding of a univalent map defined on a bounded strongly convex domain into some Loewner chain is the subject of our second result. We need following definitions before we present the statement.
	Let $D \Subset \cn$ be a domain containing the origin and 
 $$
 \mathcal{S}(D):=\{f:D\to \cn: f(0)=0, df(0)=I_{n}, f ~~ \text{is univalent} \}.
 $$
 \begin{definition}
Let $d \in [1, \infty]$. A family  of mappings $f_{t}\colon D \to \cn $ is called $L^{d}$-normalized Loewner chain on $D$ if 
 
 \begin{itemize}
     \item [i.]
    For each fix $t \geq 0$, $f_{t} \colon D \to \cn$ is an univalent holomorphic mapping such that $f_{t}(0)=0$ and $df_{t}(0)=e^{t}I_{n}$.
    \item [ii.]
 For $0\leq s <t<\infty$,  $f_{s}(D) \subset f_{t}(D)$
 \item [iii.]
 for any compact set $K\St M$ and any $T > 0$, there exists a function $\kappa_{K,T} \in L^{d}([0,1], [0,\infty))$ such that 
 such that for all $z \in K$ and for all $0 \leq s \leq  t \leq  T$ we have 
 $\|f_{s}(z)-f_{t}(z)\| \leq \int_{s}^{t}\kappa_{K,T}(x)\,dx$.
 \end{itemize}
 \end{definition} Loewner range of Loewner chain  is defined by  biholomorphism class of $R(f_{t}):=\cup_{s \geq 0} \Om_{s}$. 
	A function  $f \in \mathcal{S}(D)$ is said to be embedded into a $L^{d}$-normalized Loewner chain if there exists a $L^{d}$-normalized Loewner chain $(f)_{t}$ such that $f_{0}=f$. Here we put  some notations.
	\begin{align*}
		\mathcal{S}^{1}(D)&:=\{f \in \mathcal{S}: f~~\text{embeds into a normalized Loewner chain $(f)_{t}$}\}\\
		\mathcal{S}^{0}(D)&:=\{f \in \mathcal{S}^{1}: \{e^{-t}f_{t}\}_{t \geq 0} ~~\text{is a normal family} \}\\
		\mathcal{S_{R}}(D)&:=\{f \in \mathcal{S}: f(D)~~\text{is a Runge domain} \}
	\end{align*}
 For the unit disc $\disc\subset\cplx$, $\mathcal{S}^{0}(\disc)=\mathcal{S}^{1}(\disc
 )=\mathcal{S}(\disc)$, but,
	 for $n \geq 2$, the following chain of inclusions holds for $D=\mathbb{B}^n$: $$\mathcal{S^{\circ}}(\mathbb{B}^{n}) \subsetneq \mathcal{S}^{1}(\mathbb{B}^{n}) \subsetneq \mathcal{S}(\mathbb{B}^{n}).$$ The class $\mathcal{S}^{0}(\mathbb{B}^{n})$ is compact in the topology of uniform convergence on compact subsets of $\mathbb{B}^{n}$ but $\mathcal{S}(\mathbb{B}^{n}), ~\mathcal{S}^{1}(\mathbb{B}^{n})$ are non compact. Hence, in  higher dimension, $\mathcal{S}^{0} \subsetneq \mathcal{S}^{1}(\mathbb{B}^{n})$ (see \cite[Section 8]{GKbook}). In \cite{FoarnaessWold2020}, it is shown that $\mathcal{S}^{1}(\mathbb{B}^{n}) \subsetneq \mathcal{S}(\mathbb{B}^{n})$. Recently,   Bracci-Gumenyuk \cite{Bracci2022} showed that  $\mathcal{S_{R}}(\mathbb{B}^{n}) \subsetneq \mathcal{S}^{1}(\mathbb{B}^{n})$. In \cite[Definition 1.1]{ABW2015}, Arosio-Bracci-Wold introduced the the notion of filtering normalized Loewner chain on $\mathbb{B}^{n}$. We mention it here for any bounded domain. 
	\begin{definition} \label{D:filter}
		Let $D \Subset \cn$ and  $(f_{t})$ be normalized Loewner chain in $D$. We say that $(f_{t})$ is filtering normalized Loewner chain provided the family $\Om_{t}:=f_{t}(D)$ satisfies the following conditions.
		\begin{enumerate}
			\item [1.]
			$\overline{\Omega}_{s} \subset \Om_{t}$ for all $t >s$; and 
			\item [2.]
			for any open set $U$ containing $\overline{\Omega}_{s}$ there exist $t_{0}>s$ such that $\Om_{t} \subset U$ for all $t \in (s, t_{0})$.
		\end{enumerate}
	\end{definition}
 Let 
\begin{multline*}
  \mathcal{S}_{\mathfrak{F}}^{1}(D)=\{f \in \mathcal{S}(D): f~\text{embeds into filtering normalized Loewner chain,} R(f_{t})=\cn\}.
\end{multline*}
 
\noindent The connection between an univalent function to be embedded in a filtering normalized Loewner chain and the polynomial convexity of the closure of the image under that function was first explored by Arosio-Bracci-Wold\cite{ABW2015}. They proved the following result for $\Om \Subset \cn $ be a bounded pseudoconvex domain with $\smoo^{\infty}$ boundary, which is  biholomorphic to the open unit ball.
 
\begin{result}\cite[Theorem 1.2]{ABW2015} \label{T:ABW}
		Let $n \geq 2 $ and let $f \in \mathcal{S}_{\mathcal{R}}$.  Assume that $\Om :=f(\mathbb{B}^{n})$ is bounded strongly pseudoconvex domain with $C^{\infty}$ boundary. Then $f \in \mathcal{S}_{\mathfrak{F}}^{1}$ if and only if $\overline{\Omega}$ is polynomially convex. 
	\end{result}
\noindent	In the same article, Arosio-Bracci-Wold deduced that $\overline{\mathcal{S}_{\mathfrak{F}}^{1}(\mathbb{B}^{n})}=\mathcal{S_{R}} (\mathbb{B}^{n})$ as a corollary of \Cref{T:ABW} (see \cite[Corollary 3.3]{ABW2015}).
	Our second result in this article  we replace the unit ball with a strongly convex domain.
	\begin{theorem}\label{T:loewner}
		Let $0 \in D \Subset \cn$ be a strongly convex domain with $C^{m}$ boundary  and $f\colon D \to f(D)$ be a biholomorphism. Assume that $\Om:=f(D)$ is  bounded strongly pseudoconvex domain with $C^{m}$ boundary for some $m>2+\frac{1}{2}$. Then f can be embedded into a filtering $L^d$-Loewner chain with Loewner range $\cn $  if and only if  $\overline{f(D)}$ is polynomially convex.
	\end{theorem}	
	\noindent It also follows from  \Cref{T:loewner} and Andersen-Lempert theorem (\cite[Theorem 2.1]{AnderLemp}) that, for any bounded strongly convex domain $D \Subset \cn $,
	  $\overline{\mathcal{S}_{\mathfrak{F}}^{1}(D)}=\mathcal{S_{R}} (D)$ (See \Cref{C:cor1}). In order to prove \Cref{T:loewner}, we proved the following theorem which might be of independent interest. 	
	\begin{theorem}\label{P:global Narshiman}
		Let  $\Om \Subset \cn $ be a strongly pseudoconvex domain with $C^{k}$ boundary  which is biholomorphic to some bounded strongly convex domain with $C^{k}$ boundary for some $k>2+\frac{1}{2}$. Then the following are equivalent.
		\begin{enumerate}
			\item[1.]
			$\overline{\Om}$ is polynomially convex.
			\item[2.]
		$\exists~ \Psi \in \text{Aut}(\cn)$ such that	$\Psi(\Om)$ is strongly convex.
		\end{enumerate}
		Moreover, if one of the conclusions holds then $\Om$ is a Runge domain.
	\end{theorem}
\begin{remark}
\Cref{P:global Narshiman} is a general version of \cite[Proposition 3.6]{ABW2015} and  \cite[Proposition 3.1]{hamada2020}.
 \end{remark}
 
\noindent The following corollary provides a class of strongly pseudoconvex domains that are biholomorphic to strongly convex domains through automorphisms of $\cplx^n$.
\begin{corollary}
    Let $\Om \Subset \cn$ be a strongly pseudoconvex domain with $\smoo^{\alpha}$ boundary which is biholomorphic to a strongly convex  domain with $\smoo^{\alpha}$ boundary for $\alpha >2+\frac{1}{2}$ and $\Om$ is spirallike  with respect to a globally asymptotically stable vector field. Then there exists $\Psi \in Aut(\cn)$ such that $\Psi(\Om)$ is strongly convex.
\end{corollary}

\subsection{Dense holomorphic curves} Next, we will demonstrate an application of \Cref{T:Polynostrict} in the context of finding a dense  holomorphic map and constructing universal mapping.  For any two complex manifolds $X, Y$, the set of all holomorphic maps from $X $ to $Y$ is denoted by $\mco(X, Y)$. If $Y$ is $\cplx$ then the set of all  holomorphic functions on $X$ is denoted by $\mco(X)$. Let $Y$ be a complex manifold. The main question here is: {\em For a given complex manifold $Z$, does there exists  a holomorphic map $f\colon Z \to Y$ such that $f(Z)$ is dense in $Y$?} In this case, we say that $f $ is a dense holomorphic map from $Z$ to $Y$. In \cite{W2005}, Winkelmann proved that if $X$ and $Y$ are irreducible complex spaces and $X$ admits a non-constant bounded holomorphic function then there exists a holomorphic map from $X$ to $Y$ with dense image (for the notions of complex manifold and complex space see \cite{Forstbook}). In \cite{FW2005}, Forstneri\v{c} and  Winkelmann  showed that if $X$ is a connected complex manifold then the set of all holomorphic maps $f\colon \D \to X$  with $\overline{f(\D)}=X$ is dense in $\mathcal{O}(\D, X)$ with respect to the compact open topology. In this paper we consider the holomorphic maps with value in $\hol(X,Y)$, where $X$ and $Y$ are complex manifolds.
Let $X, Y, Z$ be connected complex manifolds and $\mathcal{S} \st \hol(X,Y)$.  We say that a map $f\colon Z \to \mathcal{S}$ is holomorphic if the map $\hat{f}\colon Z  \times X\to Y$ defined by  $\hat{f}(z,x)= f(z)(x)$ is holomorphic. In this case, $\hat{f}$ is said to be associated holomorphic map for $f$. Following the terminology introduced by Kusakabe \cite{kus2017}, we say a subset $\mathcal{S} \st \mathcal{O}(X,Y)$ is $Z$-dominated if there exists a dense holomorphic map $f\colon Z \to \mathcal{S}$.
In \cite{kus2017}, Kusakabe proved the following result.
\begin{result}\cite[Theorem 1.1]{kus2017} 
\label{R:kusu1}
Let $\Om \Subset \cn$ be a bounded convex domain and $Y$ be a connected complex manifold. Then $\mco(\Om,Y)$ is $\mathbb{D}$-dominated.     
\end{result}

\noindent By demonstrating an example \cite[Example 2.2]{kus2017}, Kusakabe showed also  that \Cref{R:kusu1} is not true  in general for bounded pseudoconvex domain.
In this paper, we use \Cref{T:Polynostrict} to provide a class of pseudoconvex  domains $\Om$ in $\cn$ for which $\mco(\Om, Y)$ is $\mathbb{D}$-dominated. Our next theorem reads as 
\begin{theorem}\label{T:densemap1}
Let $\Om \Subset \cn$ be a bounded pseudoconvex domain containing the origin and is strictly spirallike with respect to  globally asymptotic stable vector field $V \in \mathfrak{X}_{\mathcal{O}}(\cn)$ and $Y$ be a connected complex manifold. Then $\mco(\Om,Y)$ is $\mathbb{D}$-dominated.     
\end{theorem}

\subsection{Universal mappings and composition operators} We also apply \Cref{T:densemap1}  to obtain a result providing the existence of a universal mapping. In \cite{Brikhoof}, Birkhoff constructed for every sequence of real number $\{b_{k}\}_{k \in \mathbf{N}}$ with $\lim_{k \to \infty} b_{k}=\infty$, a holomorphic function  $F \in \mco(\cplx)$, with the property that if $F_{k} \colon \cplx \to \cplx$ defined by $F_{k}(z)=F(z+b_{k})$ then  $\{F_{k}| k \in \mathbb{N}\}$ is dense in $\mco(\cplx)$ with respect to the compact open topology.  Such a function is called $universal function $. Later, Seidel and Walsh \cite{Siedel1941} proved the result for the unit disc replacing the Euclidian translation with a non-Euclidian translation. In \cite[Theorem 5, 6]{Fernando99}, Fernando proved that any compactly divergent sequence of automorphisms of the open unit ball and the polydisc admits a universal function.
Existence of universal function has a very close connection with the hypercyclicity of a composition operator (See \cite{Erdman99} for a nice survey in this topic). 



   


Let  $T \colon X \to X$ be a self-map on  a topological vector space $X$,  and $(n_{k})_{k \in \mathbb{N}}$ is an  increasing sequence of  natural numbers. $T$ is said to be  \textit{hypercyclic with respect to}  $(n_{k})$ if  there exists $ x \in X$ such that $\{T^{n_{k}}(x)| k \in \mathbb{N}\}$ is dense in $X$, where  $T^{m}:=\underbrace{T \circ T \circ \cdots \circ T}_{\text{m times}}$ for every $m \in \mathbb{N}$ .

In \cite{Zajac16}, Zaj\c{a}c proved the following result: 
\begin{result}\label{R:ZZ}
    Let $\Om$ be a connected Stein manifold, and $\phi \in \mco(\Om, \Om)$ and let $(n_{k})_{k} \st \mathbb{N}$ be an increasing sequence. Then the composition operator $C_{\phi} \colon \mco(\Om) \to \mco(\Om)$ defined by $C_{\phi}(f)=f \circ \phi$ is hypercyclic  with respect to $(n)_{k}$ if and only if $\phi$ in injective and for every  compact $\mco(\Om)$ convex subset $K \st \Om$ there exists such that $K \cap \phi^{n_{k}}(K) \neq \emptyset$ and the set $K \cup \phi^{n_{k}}(K)$ is $\mco(\Om)$ convex.
\end{result}

Motivated by the properties of $\phi$ in the \Cref{R:ZZ}, Andrist and Wold \cite{AW2015}, defined $generlized ~~translation$ which is as follows : 
\begin{definition}
        Let $X$ be a Stein space and $\tau \in Aut(X)$. The automorphism $\tau$ is a generalized translation if for any compact $\mathcal{O}(X)$-convex subset $K \subset X$ there exists $j \in \mathbb{N}$ such that 
        \begin{itemize}
            \item [1.]
            $\tau^{j}(K) \cap K =\emptyset$, and
            \item [2.]
            $\tau^{j}(K) \cup K $ is $\mathcal{O}(X)$-convex.
        \end{itemize}
    \end{definition}

\noindent In \cite{AW2015}, it is proved that if $X$ is a stein manifold with density property (see \cite[Section-4.10]{Forstbook} for the notion of density property), and $\tau \in Aut(X)$ is a generalized translation, then there exists $F \in Aut_{0}(X)$ such that  the subgroup generated by $\tau $ and $F$ is dense in $Aut_{0}(X)$ in compact open topology, where $Aut_{0}(X)$ is path connected component of identity automorphism.
    
In \cite{kus2017}, Kusakabe first studied the hypercyclicity  of the composition from  the space $\mco(\Om, Y)$, where $Y$ is a connected complex manifold.

 
\begin{definition}
    Let $X$ be a complex manifold and $\tau \in Aut(X)$ and assume that $\mathcal{S} \subset \mco(X, Y)$ is a $\tau^{*}$ invariant subset, i.e. $\tau^{*} \mathcal{S} \subset \mathcal{ S}$, where $\tau^{*} \colon \mco(X,Y) \to \mco(X,Y)$ defined by $\tau^{*}(f)=f\circ \tau$. A holomorphic map $F \in \mathcal{S}$ is called an $\mathcal{S}$-universal map for $\tau$ if $\{F \circ \tau^{j}\}_{j}$ is dense in $\mathcal{S}$.
    \end{definition}
\noindent From the   above-mentioned result due to Brikhoff, for  every translation mapping $\tau(z)=z+a$, with $a \neq 0$, there exists an $\mco(\cplx)$-universal map for $\tau$. In \Cref{R:ZZ}, we had  the existence of universal map when $X=\Om$, a Stein manifold, and  $Y=\cplx$, and $\tau \in \mco(\Om, \Om) $ is generalized translation.  


 As an application of \Cref{R:kusu1}, Kusakabe proved that 
 \begin{result} \cite[Theorem 1.3]{kus2017}\label{R:kusu2}
Let $\Om \St \cn$  be a bounded convex domain, $\tau \in Aut(\Om)$ is a generalized translation, and $Y$ be a  connected complex manifold. Then there exists an $\mco(\Om, Y)$-universal map for $\tau$.  
 \end{result}
The Carath\'eodory pseudodistance is defined as follows. We will need this definition to discuss our next result.
 \begin{definition}
        Let $\Om \st \cn$ be a domain and $\mathbb{\rho}$ denotes the Poincar\'e distance in $\D$. The Carath\'eodory pseudodistance between $z,w \in \Om$ is denoted by $c_{\Om}(z,w)$ and is defined by $$c_{\Om}(z,w)=\sup\{\rho(f(z),f(w))|f \in \mco(\Om, \D)\}.$$
    \end{definition}
\noindent    A domain  $\Om$ is  said to be Carath\'eodory hyperbolic if $(\Om, c_{\Om})$ is a metric space. A $c_{\Om}$-ball centered at $x\in \Om$ and with radius $r>0$ is defined by $B_{c_{\Om}}(x,r)=\{z\in \Om| c_{\Om}(x,z)<r\}$. A Carath\'eodory hyperbolic domain $\Om$ is said to be $c_{\Om}$-finitely compact if all  $c_{\Om}$-balls  with center in $\Om$ and finite radius,  is  relatively compact in $\Om$ with respect to the usual topology of $\Om$.

\noindent As an application of \Cref{T:densemap1}, we are also able to extend \Cref{R:kusu2}, for bounded  pseudoconvex domains which are strictly spirallike with respect to a globally asymptotic stable vector field and $c$-finitely compact. More precisely, our next theorem is 

\begin{theorem}\label{T:densemap2}
 Let $V \in \mathfrak{X}_{\mathcal{O}}(\cn)$ be a complete globally asymptotic stable vector field. Let $\Om \Subset \cn$ be a bounded  pseudoconvex domain containing the origin. Suppose that $\Om$ is $c_{\Om}$-finitely compact and strictly spirallike domain with respect to the vector field $V$. Then for any generalized translation $\tau \in Aut(\Om)$ and any connected complex manifold $Y$, there exists an $\mco(\Om, Y)$-universal map for $\tau$. In particular, the conclusion of the theorem is true if $\Om$ is a strongly pseudoconvex domain with $\smoo^{2}$ boundary and spirallike with respect to the vector field $V$.
 \end{theorem}

 \begin{remark}
      \Cref{T:densemap2} and \Cref{R:kusu1} can be interpreted as the composition operator $C_\tau\colon\hol(\Om,Y)\to \hol(\Om, Y)$
     is hypercyclic for corresponding domains in those results. 
 \end{remark}

\noindent Our theorems  extend Kusakabe's results. For example, $\Omega=\{(z_1,z_2)\in \mathbb{C}^{2}\;|\;|z_{1}|<5,~|z_2|<e^{-|z_{1}|})\}$ is a bounded pseudoconvex domain satisfying the hypothesis of \Cref{T:densemap1}, \Cref{T:densemap2}. The domain $\Om$ is not convex (See \Cref{E:ex2}).

 It seems  from the above discussion that it is very important to know which are the generalized translation  on a given domain $\Om$. In this regard, we state our next result:
\begin{theorem}\label{T:densemap3}
    Let $\Om \St \cn$ be a bounded strongly pseudoconvex domain with $\smoo^{3}$ boundary containing the origin and strictly spirallike with respect to complete globally asymptotic stable vector field $V \in \mathfrak{X}_{\mathcal{O}}(\mathbb{\cplx}^{n})$. Then for an automorphism $\tau \in Aut(\Om)$ the following are equivalent.
    \begin{itemize}
        \item [1.]
        $\tau$ is a generalized translation.
          \item [2.]
          $\tau$ has no fixed point.
          \item [3.]
          $\{\tau^{j}\}_{j \in \mathbb{N}}$ is compactly divergent.
    \end{itemize}
\end{theorem}
For bounded convex domain this result was proved in \cite[Proposition 4.3]{kus2017}. We extend that to a certain class of strongly pseudoconvex domains.

\begin{remark}
   If $V \in \mathfrak{X}_{\mathcal{O}}(\cn)$  is a globally asymptotic stable vector field on $\cn$, with a non-zero equilibrium point and  the given domain is spirallike with respect to the vector field $V \in \mathfrak{X}_{\mathcal{O}}(\cn)$ containing the equilibrium point of the vector field $V$, then  \Cref{T:Polynostrict}, \Cref{T:polystrict1}, \Cref{T:densemap1}, \Cref{T:densemap2}, \Cref{T:densemap3} are also true.  
\end{remark}

	\section{Preliminaries and technical results}\label{S:pre}
	

 For any subset $A \subset \cn$, we denote $B(A,r):=\{z \in \cn: \text{dist}(z, A)<r \}$, where $\text{dist}(z, A):=\inf_{a \in A}\|z-a\|$. For a $\smoo^{2}$ function $\rho\colon \re^{n} \to \mathbb{R}$ we denote $\nabla \rho(x)=\big(\frac{\partial \rho(x)}{\partial x_{1}},\frac{\partial \rho(x)}{\partial x_{2}}, \cdots, \frac{\partial \rho(x)}{\partial x_{n}}\big)$ and Hessian matrix of the function $\rho$ at the point $x \in \re^{n}$ is denoted by $H(\rho)(x)$. 
	
  For a real number $k >0$ and  $D \subset \cn$ we say that a mapping  $f \in C^{k}(D)$ if $f$ has continuous partial derivatives of order $[k]$ on $D$ and the partial derivatives of
	order $[k]$ are H$\ddot{o}$lder continuous with exponent $k-[k]$ on $D$. 
 
 For any two complex manifolds $X, Y$, the space $\mco(X, Y)$ equipped with compact open topology, forms a second countable and completely metrizable space. In particular, it is separable Baire space (see \cite[Remark 1.1]{kus2017}).

Let $M, N$ be two  smooth manifolds and $r \in \mathbb{N}$. Suppose  that  $\smoo^{r}(M,N)$ denotes the  set of all $r$ times differentiable map. 
Let $f \in \smoo^{r}(M,N)$. Suppose that $(U, \phi), (V, \psi)$ be two charts on $M, N$ respectively. Let $K \st U$ be a compact subset such that $F(K) \st V$. For $\eps >0$ we denote  a weak subbasic neighborhood of $f$ by $\mathcal{N}^{r}(f, (U, \phi), (V, \psi), K , \eps)$ and it is defined by $\{g \in \smoo^{r}(M,N)|g(K) \st V, \sup_{x \in K}\|D^{k}(\psi f \phi^{-1}(x))-D^{k}(\psi g \phi^{-1}(x))\|<\eps, k \in \{0,1,2, \cdots ,r\}\}$. The weak topology generated by  sets $\mathcal{N}^{r}(f, (U, \phi), (V, \psi), K , \eps)$ is called compact open $C^{r}$ topology and it is  denoted  by $\smoo^{r}_{W}(M, N)$.
The next result is well known (see \cite[Exercise 7]{hirschbook}). We will use it for proving \Cref{P:global Narshiman}.

\begin{result}\label{R:hrsh exr}
		A $\smoo^{1}$ immersion $f\colon M \to N$, which is injective on a closed set $K \subset M$, is injective on a neighborhood of $K$. Moreover, $f$ has a neighborhood $\mathcal{N} \subset \smoo^{1}_{S}(M,N)$ and $K$ has a neighborhood $U$ of $M$ such that every $g \in \mathcal{N}$ is injective on $U$. If $K$ is compact then $\mathcal{N}$  can be taken in $\smoo^{1}_{W}(M, N)$.
	\end{result}

	We need the following result for proving \Cref{P:global Narshiman}.
	\begin{result}\cite[Lemma]{Krantzbook}\label{R:strongconvex}
		Let $\Om \subset{R}^{n}$ be a strongly convex domain. Then there exist a constant $c>0$ and a defining function $\tilde{\rho}$ for $\Om$ such that 
		\begin{align}\label{E:strongly1}
			\sum_{j,k =1}^{n} \frac{\partial ^2 \tilde{\rho}}{\partial x_{j} \partial x_{k}}(P)w_{j}w_{k} \geq c \|w\|^2,
		\end{align} 
		$ \forall P \in \partial \Omega$ and $\forall w \in \mathbb{R}^{n}$.
	\end{result}

We need the following notion for proving \Cref{T:Polynostrict} (see \cite[Page-76]{Rangebook} for detail). 
 \begin{definition} 
     A compact subset $K \st \cn$ is said to be Stein compactum if for every neighborhood $U$ of $K$ there exists a domain of holomorphy $V_{U} \st \cn$ such that $K \subset V_{U} \subset U$.
 \end{definition}

\noindent For any compact subset $K$ in $\cn$, $\mathcal{P}(K)$ denotes the set of all continuous functions that can be approximated by holomorphic polynomials uniformly on $K$ and $\hol(K)$ denotes the set of all continuous functions on $K$ that can be approximated uniformly on $K$ by holomorphic functions in a neighborhood of $K$.  The following result from Range's book will be   used in our proof of  \Cref{T:Polynostrict}.
 \begin{result}\cite[Theorem-1.8]{Rangebook}\label{R:steincompacta}
	Let $K \st \cn$ be a Stein compactum and assume that $\mathcal{O}(K) \subset \mathcal{P}(K)$. Then $K$ is polynomially convex.
\end{result}

Recall that a domain $U \subset \cn $ is said to be Runge if every $f \in \mathcal{O}(U)$ can be approximated by holomorphic polynomials. Next, we state a result that will be used to prove \Cref{T:Polynostrict}
 \begin{result}\cite[Theorem 1.5]{CG}\label{R:SG}
    Let $F \in \mathfrak{X}_{\mathcal{O}}{(\mathbb{C}^{n})}$ be a complete globally asymptotically stable vector field.  If $\Om \ni 0$ is a spirallike domain with respect to $F$ containing the origin then $\Omega$ is a Runge domain.
	
 \end{result}

 The following result is 
from \cite[Theorem 3.2]{For2004}, related to holomorphic approximation. We  will use this in the proof of \Cref{L:polyapp}.

 \begin{result}[Forstneri\v{c}, \cite{For2004}]\label{R:for04}
    Let $K_{0}$ and $S= K_{0}\cup  M$ be compact holomorphically convex subsets in a complex manifold $X$ such that $M=S\setminus K_{0}$ is  a totally real $m$ dimensional submanifold of class $\smoo^{r}$. Assume that $r \geq \frac{m}{2}+1$
and let $k$ be an integer satisfying $0 \leq k \leq r-\frac{m}{2}-1$. Given an open set $U \st X$ containing $K_{0}$ and a map $f \colon U \cup M \to Y$ to a complex manifold $Y$ such that 
$f\big|_{U}$ is holomorphic $f\big|_{M} \in \smoo^{k}(M)$, there exist open sets 
$V_{j} \st X$ containing $S$ and holomorphic map $f_{j}\colon V_{j} \to Y$ $(j=1,2,, \cdots )$ such that, as $j \to \infty$, the sequence $f_{j}$ converges to $f$ uniformly on $K_{0} $ and
in the $\smoo^{k}$-sense on M. If in addition, $X_{0}$ is a closed complex subvariety of $X$ which does not intersect $M$ and $s \in \mathbb{N}$ then we can choose the approximating sequence such that $f_{j}$ agrees to order $s$ with $f$ along $X_{0} \cap V_{j}$ for
all $j = 1, 2, 3 \cdots $.
   \end{result}

The next result from \cite[Theorem 1.16]{Erdbook}, gives a necessary and sufficient condition for being a  continuous map hypercyclic. We will use this result to prove \Cref{L:GT2}.  

\begin{result}[Birkhoff, \cite{Erdbook}]\label{R:brik1}
  Let $T\colon X \to X $ be a continuous
map on a separable complete metric space $X$ without isolated points. Then
the following assertions are equivalent:
\begin{itemize}
    \item [i.]
 For any pair $U, V $ of  nonempty open subsets of $X$, there exists
some $ n \geq  0$ such that $T^{n}(U) \cap V \neq \emptyset$
\item [ii.]
there exists some $x \in X$ such that $\{T^{m}(x)|m \in \mathbb{N}\}$ is dense $X$.
\end{itemize}

\end{result}

The next  result from \cite[Theorem 3]{Pinchuk}, is related to the diffeomorphic  extension of a biholomorphism defined on a strongly pseudoconvex domain. It will be used in the proof of \Cref{P:global Narshiman}. 
	\begin{result}[Hurumov, \cite{Pinchuk}]\label{R: pinchuk}
		Let $D_{1}, D_{2} \Subset \cn$ be strongly pseudoconvex domain with the boundaries of class $C^{m}(m \geq 2)$, and 	$f: D_{1} \to D_{2}$ is a biholomorphism or proper holomorphic mapping. Then $f \in C^{m-\frac{1}{2}}(\overline{D_{1}})$, if $m-\frac{1}{2}$ is not an integer and $f \in C^{m-\frac{1}{2} -\eps}(\overline{D_{1}})$, for arbitrarily small $\epsilon>0$ if $m-\frac{1}{2}$ is  an integer. 
	\end{result}
	The next  result from   \cite[Theorem 24]{Foarnaess}, is a  Mergelyan type approximation on a strongly pseudoconvex domain. It  will  be used  in the proof of \Cref{P:global Narshiman}.
	\begin{result}\cite[Theorem 24]{Foarnaess}\label{R: foarnaess}
		Let $X$ be a Stein manifold and $\Om \Subset X$ be a strongly pseudoconvex domain of class $C^{k}$ for $k \geq 2$. Then for any $f \in C^{k}(\overline{\Om})\cap \mathcal{O}(\Om)$ , $k \geq 2$ there exists a sequence of function $f_{m} \in \mathcal{O}(\overline{\Om})$ such that $\lim_{m \to \infty}\|f_{m}-f\|_{C^{k}(\overline{\Om})}=0$. 
	\end{result}
 
The following  result  can be proved using \cite[Theorem 1.6.9]{Stoutbook}. We will use it in the proof of \Cref{L:polyapp}.

\begin{result}\label{L:kallin}
Let $K$ be a compact polynomially convex subset of $\cn$ containing the origin. Then $((\overline{\mathbb{D}} \cup \{2\}) \times K)\cup [1,2] \times \{0\} \subset \mathbb{C}^{1+n}$ is also polynomially convex. 
\end{result}

The next three  lemmas are the main ingredients of the proof of \Cref{P:global Narshiman} and \Cref{T:loewner}.

	\begin{lemma}\label{L:injective}
		Let $\Om \Subset \cn$ be a domain and  $f \in \smoo^{1}(\OOm , \cn)$. Suppose that  $f$ is  injective on $\OOm$ and $Df(z)$ is invertible for all $z \in \partial \Om$. Then there exists a neighborhood $U$ of $\overline{\Om}$ , such that $f$ is injective  on $U$. 
	\end{lemma}
	\begin{proof}
	  	Let $f \colon \overline{\Om} \to \cn$ be an injective $\smoo^{1}$ map.  Since $f$ is $\smoo^{1}$ on closed set, hence it is $\smoo^{1}$ on a neighborhood of $\overline{\Om}$. Assume that $f$ is not injective on any neighborhood of $\overline{\Om}$. Therefore, there exists $N \in \mathbb{N}$  and $z_{m}, w_{m} \in B(\overline{\Om}, \frac{1}{m})$,  such that $z_{m} \neq w_{m}$ and $f(z_{m})=f(w_{m})$,  for  $m >N$. Clearly,  $\{z_{m}\}$ and $\{w_{m}\}$ are two bounded sequence. Hence, passing to subsequence we can assume that  there exist  $z_{0}, w_{0} \in \overline{\Om}$ such that $z_{m} \to z_{0}$ and $w_{m} \to w_{0} $, as $m \to \infty$.
	 
	  Clearly, $\lim_{m \to \infty}f(z_{m})=\lim_{m \to \infty}f(w_{m})$. Hence, $f(z_{0})=f(w_{0})$. Since $f$ is injective on $\overline{\Om}$, hence we have $z_{0}=w_{0}$. Now from the inverse function theorem, it follows that $f$ is a local diffeomorphism at $z_{0}$. Since  every neighborhood of $z_{0}$, contains two distinct points $z_{m}$ and $w_{m}$  such that $f(z_{m})=f(w_{m})$, hence $f$ can  not  be a locally injective map  at $z_{0}$. Therefore, we get a contradiction. This proves the lemma.
		
			\end{proof}

	\begin{lemma}\label{L:L1}
		Let $\Om \subset \cn$ be a bounded domain and $U$ be a  Runge domain in $\cn$ containing $\overline{\Om}$. If $h: U \to h(U)$ is a biholomorphism such that $h(U)$ is convex and $h(\Om)$ is strongly convex with $\smoo^{2}$ boundary then  there exists $\Psi \in Aut(\cn)$ such that  $\Psi(\Om)$ is a strongly convex domain.  	
	\end{lemma}
	\begin{proof}

		Let $U \steq \cn $ be a Runge domain. Assume that $h\colon U \to h(U)$ is a  biholomorphism such that $h(U)$ is a convex  domain. We now invoke Anders\'en-Lempert theorem \cite[Theorem 2.1]{AnderLemp}, to get that $h^{-1}:h(U) \to U$ can be approximated by $Aut(\mathbb{C}^{n})$ uniformly over every compact subset of $h(U)$. Let $\psi_{m}^{-1} \in Aut(\cn)$  such that $\psi_{m}^{-1}$ converges to $h^{-1}$ uniformly on every compact subset of $h(U)$. Then $\psi_{m}$ converges to $h$ uniformly on every compact subset of $U$. We prove that $\psi_{m}(\Om)$ is strongly convex for  large enough $m \in \mathbb{N}$. From \Cref{R:strongconvex}, we obtain a  defining function  $\rho: \mathbb{R}^{2n} \to \mathbb{R}$ of the  domain $h(\Om)$ and a constant $C>0$, such that the following holds: 
		\begin{align}\label{E:strongly2}
			\sum_{j,k =1}^{2n} \frac{\partial ^2 \rho}{\partial x_{j} \partial x_{k}}(P)w_{j}w_{k} \geq C \|w\|^2, \hspace{.5cm} \forall P \in \partial h(\Om) \hspace{.3cm}\forall w \in \mathbb{R}^{2n}.
		\end{align}
  \noindent
Clearly,  $\rho \circ h: U \to \re $ is a $\smoo^{2}$ defining function for $\Om$. Let $V', V'' \subset h(U)$ be such that $h(\OOm) \Subset V' \Subset V'' \Subset h(U)$. Now $\psi_{m}^{-1} \to h^{-1}$ uniformly over $\overline{V''}$. Thus, there exists $m_{0} \in \mathbb{N}$  such that $\psi_{m}^{-1}(\overline{V'}) \Subset h^{-1}(V'') \Subset U$. Therefore, the function $\tilde{\rho}_{m}\colon V' \to \re$ defined by $\tilde{\rho}_{m}(z)=\rho \circ  h \circ \psi_{m}^{-1}(z)$ is 
 well defined and a defining function for $\psi_{m}(\Om)$,  for all $m>m_{0}$.  Since $\psi_{m}^{-1}$ converges to $h^{-1}$ on  every compact subset of $h(U)$, hence  $h \circ \psi_{m}^{-1} \to id_{V'}$  uniformly on every compact subset of $V'$. Consequently, $\tilde{\rho}_{m} \to \rho$  locally uniformly on $\overline{h(\Om)}$. 
 
 Clearly, $\nabla\tilde{\rho}_{m}(z)=\nabla \rho (h(\psi_{m}^{-1}(z)))D(h\circ\psi_{m}^{-1})(z)$. Since $h \circ \psi_{m}^{-1}$ is holomorphic map, hence $D(h\circ\psi_{m}^{-1})(z)$ converges to $I_{n}$ uniformly on every compact subset of $V'$. Since $\nabla \rho$ is continuous on $V'$, hence $\nabla\rho_{m}(z) \to \nabla\rho(z)$ locally uniformly on $V'$.  From the chain rule of second order derivative, we get that  
 
 \begin{align*}
 H(\rho_{m})(z)&=D(h\circ\psi_{m}^{-1}(z))^{T}H(\rho)(h\circ \psi_{m}^{-1}(z))D(h\circ\psi_{m}^{-1}(z))+\\
 &\sum_{j=1}^{2n}\frac{\partial \rho((h\circ \psi_{m}^{-1}(z)))}{\partial x_{j}}H(\Pi^{j}((h\circ \psi_{m}^{-1}))(z),\end{align*} 
		\noindent
		where $\Pi^{j} \colon \mathbb{R}^{2n} \to \mathbb{R}$ is projection on $j$ th component. Since $\rho$ is a $\smoo^{2}$ function, hence $H(\rho)(h\circ \psi_{m}^{-1}(z))$ converges to $H(\rho)(z)$ locally uniformly on $V'$. Since $h \circ \psi_{m}^{-1}$ is holomorphic for all $m \in \mathbb{N}$, hence, $D(h\circ\psi_{m}^{-1})(z) \to I_{n}$ and  $H(\Pi^{j}((h\circ \psi_{m}^{-1}))(z) \to O$  as $m \to \infty$  locally uniformly on $h(U)$. Therefore, we conclude that  $H(\tilde{\rho}_{m}) \to H(\rho)$ uniformly on every compact subset of $V'$, particularly on $\partial{h(\OOm)}$. 

  Let us define a map $F_{m} \colon \partial{h(\OOm)} \times S^{2n-1} \to \mathbb{R}$ by $$F_{m}(p,x)=x^{T}(H(\tilde{\rho}_{m})(p)-H(\rho)(p))x,$$
 where $S^{2n-1}:=\{z \in \cn|\|z\|=1\}$.  
 Since we have $\|H(\tilde{\rho}_{m})(p)-H(\rho)(p))\|\to 0$ uniformly on $\partial{h( \Om)}$, hence $F_{m}(p,x) \to 0$ uniformly over $\partial{h(\OOm)} \times S^{2n-1}$.
 
 Consequently, there exists $ N \in \mathbb{N}$, such that   for all $m >N$, for every  $p \in \partial h(\Om)$ and  for every non zero $(w_{1},w_{2}, \cdots ,w_{2n}) \in \mathbb{R}^{2n} $  we have 
  $$\bigg(\frac{w}{\|w\|}\bigg)^{t}D^{2}\tilde{\rho}_{m}(p)\bigg(\frac{w}{\|w\|}\bigg)>\bigg(\frac{w}{\|w\|}\bigg)^{t}H\rho(p)\bigg(\frac{w}{\|w\|}\bigg)-\frac{C}{2}=\frac{C}{2}.$$ 
 Here $\tilde{\rho}_{m}$ is a $\smoo^{2}$ smooth function. Hence, there exist $r_{p}>0, C'>0$,   such that for all $q \in B(p, r_{p})$, for all non zero  $w 
		\in \mathbb{R}^{2n}$, for all $m>N$, we have  $$\bigg(\frac{w}{\|w\|}\bigg)^{t}H(\tilde{\rho}_{m})(q)\bigg(\frac{w}{\|w\|}\bigg)>C'.$$   Since $\partial h(\Om)$ is a compact subset, hence we conclude that  there exists $\eps>0$ such that  $$\bigg(\frac{w}{\|w\|}\bigg)^{t}H(\tilde{\rho}_{m})(p)\bigg(\frac{w}{\|w\|}\bigg) >C',$$ for all $p \in B(\partial h(\Om), \eps)$ and for all $w \in S^{2n-1}$. We choose $m_{1} \in \mathbb{N}$ such that  $\partial \psi_{m}(\Om) \subset B(\partial h(\Om), \eps)$, for all $m >m_{1}$. Therefore,  $\psi_{m}(\Om)$ is strongly convex  $\forall m >\max\{m_{0}, m_{1}, N\}$. 
 
	\end{proof}
	\begin{lemma}\label{L:convexify}
		Let $\Om \Subset \cn$ be a strongly convex domain with $\smoo^{2}$ boundary. Suppose that $f_{m} \colon \overline{\Om} \to \cn$  with $f_{m}\in \mco(\Om, \cn)\cap \smoo^{2}(\overline{\Om})$ is a  sequence of diffeomorphism such that $f_{m} \to i_{\overline{\Om}}$, uniformly over $\overline{\Om}$, in $\smoo^{2}$ topology. Then there exists $N \in \mathbb{N}$ such that $f_{m}(\Om)$ is strongly convex for  all $m>N$. 
	\end{lemma}
	\begin{proof}
		Let $\Om \Subset \cn$ be a strongly convex domain with $\smoo^{2}$, boundary. Applying \Cref{R:strongconvex}, we choose a defining function $\rho :\cn \to \mathbb{R}$ of the domain $\Om$ and $C>0$ such that for all $P \in \partial \Om$ and for all $w=(w_{1}, w_{2},\cdots ,w_{2n}) \in \mathbb{R}^{2n}$ the following holds: 
		\begin{align}\label{E:strongly3}
			\sum_{j,k =1}^{n} \frac{\partial ^2 \rho(P)}{\partial x_{j} \partial x_{k}}w_{j}w_{k} \geq C \|w\|^2.
		\end{align} 
\noindent
Here $f_{m}$ is injective on $\overline{\Om}$, and  $Df_{m}(z)$ is invertible for all $z \in \overline{\Om}$, for all $m \in \mathbb{N}$. Hence, from \Cref{L:injective}, we conclude that for all $m \in \mathbb{N}$ there exists a neighborhood of $V_{m}$ of $\overline{\Om}$, such that  $f_{m}$ is also injective on  $V_{m}$. Clearly,  $\rho\circ f_{m}^{-1}\colon f_{m}(V_{m}) \to \mathbb{R}$ is a defining function of the domain $\Om_{m}:=f_{m}(\Om)$. 
		
		Since for every $m\in \mathbb{N}$, $f_{m}\colon \overline{\Om} \to f_{m}(\overline{\Om})$ is a diffeomorphism, hence  $\overline{\Om}_{m}=f_{m}(\overline{\Om}$) and  $\partial \Om_{m}=f_{m}(\partial \Om)$. Consequently, for all $x \in \partial \Om_{m}$ we have that $f_{m}^{-1}(x) \in \partial \Om$. Now for every $x \in \partial  \Om_{m}$ we have $Df_{m}^{-1}(x)=(Df_{m}(f_{m}^{-1}(x)))^{-1}$. Now $\|(Df_{m}(y))^{-1}-I\| \to 0$ locally uniformly over $\overline{\Om}$. Since for all $m \in \mathbb{N}$ and for all $x \in \partial \Om_{m}$  we have  $f_{m}^{-1}(x) \in \partial \Om$, hence $	Df_{m}^{-1}(x) \to I$ uniformly over $\partial \Om_{m}$, as $m \to \infty$ in $\smoo^{1}$ topology. Let $(f_{m}^{-1})^{j}$ denotes the $j$ th component of the map $f_{m}^{-1}$. Clearly, we get that $\nabla (f_{m}^{-1})^{j}(x) \to \underbrace{(0,0, \cdots 1, \cdots, 0)}_{\text{j th position}}$ as $m \to \infty$ uniformly over $\partial \Om_{m}$ in $\smoo^{1}$ norm.  Since $H((f_{m}^{-1})^{j}(x))=D(\nabla (f_{m}^{-1})^{j})(x)$, hence  $H((f_{m}^{-1}))^{j}(x) \to O$, uniformly over $\partial \Om_{m}$ as $m \to \infty$.
		Next, using the chain rule of the Hessian matrix  we obtain that  
		\begin{align}\label{E:Hesscompo}
			H(\rho\circ f_{m}^{-1})(x)&= (Df_{m}^{-1}(x))^{T}H\rho (f_{m}^{-1}(x)) (Df_{m}^{-1}(x))+\sum_{j=1}^{2n}\frac{\partial \rho}{\partial x_{j}}(f_{m}^{-1}(x))H((f_{m}^{-1})^{j})(x),
		\end{align}
  for all  $x\in \partial \Om_{m}$ and for all $m \in \mathbb{N}$.
		We have that  $H((f_{m}^{-1}))^{j}(x) \to O$ uniformly over $\partial \Om_{m}$ as $m \to \infty$. Hence, there exists $N_{1} \in \mathbb{N}$ such that for all $m >N_{1}$
		\begin{align}\label{E:hess}
			\sup_{\|w\|=1, x\in \partial D_{m}}\big|\sum_{j=1}^{2n}\frac{\partial \rho}{\partial x_{j}}(f_{m}^{-1}(x))w^{T}H((f_{m}^{-1})^{j}(x))w \big|<\frac{C}{3}.  
		\end{align}
		Therefore, taking into account \eqref{E:Hesscompo}, \eqref{E:hess}, \eqref{E:strongly3} we conclude that 
		\begin{align}
			w^{T}H(\rho\circ f_{m}^{-1})(x)w & >\frac{2C}{3},
		\end{align}
  for all $w \in \cn$, with $\|w\|=1$,  for all $x \in \partial \Om_{m}$ and for all $m>N_{1}$. 
		Therefore, from \Cref{R:strongconvex}, we conclude that  $\Om_{m}$ is strongly convex for large enough $m \in \mathbb{N}$.
		\end{proof}

	\section{Polynomial convexity and Spirallike domains}\label{S:proof of result}

We begin this section with the following lemma that will be used to prove \Cref{T:Polynostrict}. 

	\begin{lemma}\label{L:spirallike}
		Let $ \Om   \Subset{\cn}$ be a domain containing the origin and  spirallike  with respect to   $V \in \mathfrak{X}_{\mathcal{O}}(\cn)$. Then   $X(t,z) \in \overline{\Om}$,  $\forall z \in \overline{\Om}$, $ \forall t \geq 0$. Moreover, if $\Om$ has $C^{1}$ smooth boundary and $V(z) \notin T_{z} \partial \Om $ (i.e. $V(z)$ is transversal to the boundary) then  $X(t,z) \in \Om$, for all $z \in \overline{\Om}$, for all $t>0$.
	\end{lemma}
	\begin{proof}
		Let $z_{1} \in \partial \Om$. Suppose that $X(s, z_{1}) \notin \overline{ \Om}$, for some $s >0$. Then $\exists~ r>0$ such that $B(X(s,z_{1}),r)\cap  \Om=\emptyset$. Since $X_{s}\colon\cn \to \cn$ is a continuous map, hence there exist $\delta_{r}>0$, such that $X_{s}(z) \in B(X_{s}(z_{1}),r)$, for all $z \in B(z_{1},\delta_{r})$. Therefore, for all  $z \in  \Om \cap B(z_{1}, \delta_{r})$, it follows  that $X_{s}(z) \in B(X_{s}(z_{1}),r)$. Choose $w \in B(z_{1},\delta_{r}) \cap \Om $. Then  $X_{s}(w) \notin  \Om$. This contradicts the assumption that  $\Om$ is spirallike with respect to $V$. 
  
  Suppose that $V(z) \notin T_{z}\partial\Om$, for every $z \in \partial \Om$. Assume that $z \in \partial \Om$, and  $X(t_{1},z) \in \partial \Om$ for some $t_{1}>0$. Then we get that $X(t,z) \in \partial \Om$ for all $t \in [0, t_{1}]$. Now if $X(t,z) \in \partial \Om$ for $t \in [0,t_{1}]$, then $\frac{d}{dt}(X(t,z)) =V(X(t,z)) \in T_{X(t,z)}\partial \Om$ for $t \in (0,t_{1})$. This again leads to a contradiction with the fact $V(z)\notin T_{z}\partial \Omega$.
\end{proof}

	\begin{proof}[Proof of \Cref{T:Polynostrict}]
		Since polynomial convexity remains invariant under automorphism of $\cn$, hence, without loss of generality we can assume that $D$ is a strictly spirallike domain with respect to the globally asymptotic stable vector field $V \in \mathfrak{X}_{\mathcal{O}}(\cn)$. Suppose that $X\colon \mathbb{R} \times \cn \to \cn$ is the flow of the holomorphic vector field $V$.   We show that the family $\{X_{-t}(D)\}_{t \geq 0}$ forms a Runge and Stein neighborhood basis of $\overline{D}$. Since $D$ is pseudoconvex domain and $X_{t} \in Aut(\cn)$, hence,  $X_{t}(D)$ is pseudoconvex domain for every $t \in \mathbb{R}$. 

 \noindent
  Let $t>0$ and $w \in X_{-t}(D)$. Since $D$ is spirallike with respect to the vector field $V$, hence, for any $\tau >0$ we have $X(t+\tau, w) \in D$. Therefore, $X(-t, (X(t+\tau, w)))\in X_{-t}(D)$. Consequently, 
  for all  $t>0$, $X_{-t}(D)$ is spirallike  with respect to $V \in \mathfrak{X}_{\mathcal{O}}(\cn)$. We now invoke \Cref{R:SG}, to conclude that $X_{-t}(D)$ is Runge domain for all $t>0$.
  

  \noindent
Clearly, $\forall z \in \overline{D}$ we have $z=X_{-t}(X_{t}(z))$. Since $D$ is strictly spirallike, hence, we get that  $X_{t}(z) \in D$ for all $z \in \overline{D}$.  Therefore, $\overline{D} \subset X_{-t}(D)$ for all $t \geq 0$. Let $U$ be a domain in $\cn$ such that $\overline{D} \subset U \subseteq \cn$. Now for every $z \in \overline{D}$ $\exists r_{z}>0$ such that $B(z,r_{z}) \subset U$. Clearly, $\overline{D} \subseteq \cup_{z \in \overline{D}}B(z,\frac{r_{z}}{3})$. Since $\overline{D}$ is compact, hence there are $z_{1},z_{2}, \cdots ,z_{m} \in \overline{D}$ such that $\overline{D} \subset \cup_{i=1}^{m}B(z_{i}, \frac{r_{z_{i}}}{3})$ . Since $X(t,z)$ is continuous map, hence, there exists $T>0$ such that $X_{t}(z_{j}) \in B(z_{j}, \frac{r_{z_{j}}}{3})$  for all $0<t <T$ and for all $j \in \{1, 2, \cdots m\}$. Let $B:=\sup_{z \in \overline{D}} \|DV(z)\|$. Now, for all $w \in \overline{\Om}$, we have  $w \in B(z_{j}, \frac{r_{z_{j}}}{3})$ for some $z_{j} \in \overline{D}$. Let $0<T'<\min\{\frac{1}{B}\ln{\frac{3}{2}},T\}$. Applying \cite[Lemma1.9.3]{Forstbook}, we conclude that,   for all $t \in (0,T')$, the following holds: 
		\begin{align}\label{E:estimate}
			\|X_{-t}(w)-z_{j}\| &=\|X_{-t}(w)-X_{-t}(X_{t}(z_{j}))\|\notag\\
			& \leq e^{Bt}\|w-X_{t}(z_{j})\|\notag\\
			&\leq e^{Bt}\left(\|w-z_{j}\|+\|z_{j}-X_{t}(z_{j})\|\right) \notag\\
			& < e^{Bt}.\frac{2r_{z_{j}}}{3}\notag \\
   &<r_{z_{j}}.
		\end{align} 
\noindent
 Hence, for every $t \in (0,T')$, we obtain $\overline{D} \st X_{-t}(D) \st X_{-t}(\overline{D}) \subset \cup_{j=1}^{m}B(z_{j}, r_{z_{j}}) \subset U$. Therefore,  $\{X_{-t}(D)\}_{t \geq 0}$ forms a  Runge and Stein neighborhood basis of $\overline{D}$. Clearly, $\overline{D}$ is Stein compactum. Since $\OOm$ admits a Runge neighborhood basis, hence, $\mathcal{O}(\overline{D}) \subset \mathcal{P}(\overline{D})$. Therefore, from \Cref{R:steincompacta}, we conclude that $\overline{D}$ is polynomially convex.
	\end{proof}

We deduce the following corollary using \Cref{L:spirallike} and  \Cref{T:Polynostrict}.

\begin{corollary}
	Let $V \in \mathfrak{X}_{\mathcal{O}}(\cn)(n \geq 2)$ be a complete asymptotic stable vector field. Let $ D \subseteq \cn$ be  a pseudoconvex domain with $\smoo^{1}$ boundary containing the origin. Assume that $D$ is spirallike with respect to  $V$  and $V(z) \notin T_{z}\partial D$,  $\forall z \in \partial D$ (i.e. boundary is transversal to the vector field). Then $\overline{D}$ is polynomially convex. 
\end{corollary}

\begin{proof}[Proof of \Cref{T:polystrict1}]
	Let $D$ be a bounded pseudoconvex domain.   Let $\theta(t,z)$ be the flow of the vector field $V$. 
 Let $\sigma\colon [0, \infty) \to \mathbb{R}$ defined by $\sigma(t)=r(\theta(t,z))$. Clearly,  for every $z \in \overline{D}$, $r(z) \leq 0$. Now we have the following 
	\begin{align}\label{E:s1}
	\dfrac{d \sigma(t)}{\,dt}&=\sum_{j=1}^{n}\frac{\partial r (\theta(t,z))}{\partial z_{j}}\frac{(\theta(t,z))^{j}}{\,dt}+\frac{\partial r (\theta(t,z))}{\partial \overline{z_{j}}}\frac{\overline{(\theta(t,z))^{j}}}{\,dt}\notag\\
	&=2\cdot \rl\bigg(V^{j}(\theta(t,z))\frac{\partial r(\theta(t,z))}{\partial z_{j}}\bigg)<0.
    \end{align}
 
From our assumption we get that $	\dfrac{d \sigma(t)}{\,dt}<0$ for all $t \in [0, \infty)$. Hence, from \eqref{E:s1},  we have $r(\theta(t,z)) < r(z)$   $ \forall z \in \overline{D}$. Therefore, $\overline{D}$ is a strictly spirallike domain with respect to  a globally asymptotic stable vector field $V$. Therefore, applying \Cref{T:Polynostrict}, we conclude that $\overline{D}$ is polynomially convex.
	\end{proof}

\begin{proof}[Proof of \Cref{P:global Narshiman}]
	Since compact convex subsets of $\cn$ are polynomially convex  
 and polynomial convexity is invariant under $Aut(\cn)$ hence $(2) \implies (1)$. Now,  we prove that $(1) \implies (2)$.  Let  $\Om \Subset \cn$ be a strongly pseudoconvex domain with $\smoo^{m}$ boundary for some $m>2+\frac{1}{2}$, such that  $\OOm$ is polynomially convex. Suppose that $D$ is   a bounded strongly convex domain with the same boundary regularity as $\Om$ and   $f\colon  \Om \to D$ be a biholomorphism. In view of	\Cref{L:L1}, it is enough to construct the following: 
	\begin{itemize}
		\item
		A Runge neighborhood $U$ of $\overline{\Om}$ and a univalent map $h\colon U \to h(U)$, such that $h(U)$ is convex.
		\item
		$h(\Om)$ is strongly convex with  $\smoo^2$ boundary.
	\end{itemize}
	Since $D$ is a bounded strongly convex domain  with $\smoo^{2+\frac{1}{2}}$ boundary, hence,  $D$ is a bounded strongly pseudoconvex domain with the same boundary regularity.  Now applying \Cref{R: pinchuk},  we get a diffeomorphic extension of $f$ on $\OOm$. Let   $\tilde{f}\colon \overline{\Om} \to \overline{D}$ be the diffeomorphic  extension of $f$. In view of \Cref{R: foarnaess} we conclude that there exist a sequence  $\phi_{m} \in \mco(\overline{\Om}, \cn)$ such that $\phi_{m}$ converges to $f$ uniformly over $\OOm$ in $\smoo^{2}$ topology. 
	
	Here $\tilde{f}\colon \OOm \to \overline{D}$ is a  diffeomorphism (at least $\smoo^{2}$ regularity) . Therefore, from \Cref{L:injective}, we obtain a neighborhood $V'$ of $\OOm$ such that $\tilde{f}$ injective on $V'$. Shrinking  $V'$ if needed  we can assume that $D\tilde{f}(z)$ is invertible for every $z \in V'$. We consider an open set $V$ such that  $\OOm \st V \Subset V'$. Clearly, $\tilde{f}\colon V \to \tilde{f}(V)$ is a $\smoo^{1}$ immersion.

	Here $\OOm$ is a compact subset of $V$ and $\phi_{m}$ converges uniformly to the map $\tilde{f}$ on $\OOm$  in $\smoo^{2}$ topology. Hence, invoking  \Cref{R:hrsh exr}, we get a large enough $m' \in \mathbb{N}$ and a neighbourhood $U_{2}$ of $\OOm$ in $V$ such that $\phi_{m}$ is injective on $U_{2}$ for all $m >m'$. Since $\phi_{m}$ is also holomorphic on $\OOm$, hence it is injective holomorphic on a neighborhood $V'_{m}$ of $\OOm$. Hence, $\phi_{m}$ is particularly a diffeomorphism on $\OOm$ for all $m>m'$. Therefore, for all $m>m'$,   $\phi_{m} \circ\tilde{f}^{-1}\colon \overline{D} \to \cn$ sequence of diffeomorphism on $\overline{D}$ such that $\phi_{m}\circ \tilde{f}^{-1} \to i_{\overline{D}}$ in $\smoo^{2}$ topology. Hence, we infer from \Cref{L:convexify}, that there exist $m_{1} \in \mathbb{N}$ such that $\phi_{m} \circ\tilde{f}^{-1}(D)=\phi_{m}(\Om)$ is strongly convex domain for all $m >\max\{m_{1},m'\}$. Therefore, taking $h=\phi_{m}$, for any fix $m>\max\{m',m_{1}\}$, we get  a domain $U_{2}$ containing $\OOm$ such that  $h\colon U_{2} \to h(U_{2})$ is a  univalent map and $h(\Om)$ is strongly convex. Since $\overline{	\Omega}$ is polynomially convex, hence, it has a Runge neighborhood basis. Choose $U
	''$ a Runge neighborhood of $\overline{	\Omega}$ and $h(U'')$ is a neighborhood of $h(\overline{\Omega})$. Since $h(\overline{\Om})$ is a strongly convex domain, hence, it has a strongly convex neighborhood basis. Now choose  a strongly convex domain $D'$ such that $h(\overline{\Om}) \subset D' \subset h(U'')$.  Since $D'$ is Runge in $\cn$, hence, $(D', h(U''))$ is a Runge pair. Therefore, $(h^{-1}(D'), U'')$ is a Runge pair. Since $(U'', \cn)$ is a Runge pair, hence, $(h^{-1}(D'), \cn )$ is a Runge pair. Therefore, $h^{-1}(D')$ is a Runge domain. Choose, $U=h^{-1}(D')$, which implies that $h(U)$ is convex and $U$ is Runge. Therefore, we are done.
Since every convex domain is Runge (see \cite{kasimi}) and the automorphism of $\cn$ maps the Runge domain onto the Runge domain, hence conclusion (2) implies that $\Om$ is a Runge domain. From \cite[Page 233]{Wermer}, we conclude that $\OOm$ is polynomially convex implies that $\OOm^{\circ}$ is Runge. Since for every 
 convex domain $D$ we have $\overline{D}^{\circ}=D$, hence, we have  $$\OOm^{\circ}=\overline{f(D)}^{\circ}=f(\overline{D})^{\circ}=f(\overline{D}^{\circ})=f(D)=\Om$$
 Hence, $\Om$ is a Runge domain.
	
\end{proof}

\section{Loewner chains}\label{S:loewner}
 Let  $A \in Gl(n, \cplx)$. We say that a family of the univalent map $(f)_{t}$  on $D$ is an $A$-normalized Loewner chain on $D$ if $f_{t}(0)=0$ and $df_{t}(0)=e^{tA}I_{n}$, for all $t \geq 0$,  with $f_{s}(D) \subset f_{t}(D)$.
The following Proposition is a generalization of \cite[Proposition 3.5]{hamada2020}. We denote $Aut_{0}(\cn)=\{\psi\colon \cn \to \cn| ~\psi \text{~is an automorphism,}~ \psi(0)=0, D\psi(0)=I_{n}\}$. A particular version of the proposition, when the vector field $V =-I$,   will be used to prove \Cref{T:loewner}.
	\begin{proposition}\label{L:filter}
	Let $ D \Subset \cn$ be a   domain containing the origin. Let $f \colon D \to f(D) \Subset \cn$ be a univalent map such that $f \in \mathcal{S}(D)$. If there exists $\Psi \in Aut_{0}(\cn)$ such that $\Psi(f(D))$ is bounded strictly spirallike domain with respect to a complete globally asymptotic stable vector field $V \in \mathfrak{X}_{\mathcal{O}}(\cn)$ then $f$ embeds
into a filtering  $-DV(0)$- normalized $L^{d}$-Loewner chain with range $\cn$.
	\end{proposition}
	\begin{proof}
 Let $f \colon D \to f(D)$ be a univalent map such that $f \in \mathcal{S}(D)$. Here, $f(D)$ is a bounded domain  and there exists  $\Psi \in \text{Aut}_{0}(\cn)$  such that $\Psi(f(D))$  is strictly  spirallike domain with respect to the globally asymptotic stable vector field  $V$. Suppose that $X\colon \mathbb{C}\times  \cn \to \cn $ be the flow of the vector field $V$. We show that  $f\colon [0,\infty) \times D \to \cn$ defined by $f_{t}=\Psi^{-1}(X_{-t}(\Psi(f(z))))$  forms a filtering  $-DV(0)$- normalized $L^{d}$-Loewner chain with range $\cn$. For any $0 \leq s <t$, assume that  $f_{s}(D)=\Omega_{s}$. At first we show that $\OOm_{s} \st \Om_{t}$. If $w_{s} \in \overline{f_{s}(D)}$, then there exist a sequence $\{z_{sm}\}_{m \in \mathbb{N}}$ in $D$ such that  $w_{s}=\lim_{m \to \infty}\Psi^{-1}(X_{-s}(\Psi(f(z_{sm}))))$. Let $z_{tm}=f^{-1}(\Psi^{-1}(X_{(t-s)}(\Psi(f(z_{sm})))))$. Then we get that $$w_{s}=\lim_{m \to \infty}f_{t}(z_{tm})=f_{t}(\lim_{m \to \infty}z_{tm}).$$
 Clearly, $\lim_{m \to  \infty}(\Psi(f(z_{sm}))) \in \overline{\Psi(f(D))}$.  Since $\Psi(f(D))$ is strictly  spirallike domain with respect to the vector field  $V$, hence, we conclude that $X_{\tau}(\overline{\Psi(f(D))}) \subset \Psi(f(D))$ for all $\tau>0$. Hence, $X_{t-s}(\lim_{m \to  \infty}(\Psi(f(z_{sm})))) \in \Psi (f(D))$ for all $0<s <t$. Hence, we have $\lim_{m \to \infty}z_{tm} \in D$ for  all $t>0$. From this  we obtain that $\overline{\Om}_{s} \subset \Om_{t}$ for all $s<t$.
 Now let $\overline{\Om}_{s} \subset U$. Then, we have
	\begin{align*}
		X_{-s}\Psi(\overline{f(D)})& \subset \Psi(U).
	\end{align*}
	Since $X_{-s}(\Psi(\overline{f(D)}))$  is a compact subset and $\Psi(U)$ is an open subset of $\cn$, hence,  there exists $r_{s}>0$ such that $B(X_{-s}\Psi(\overline{f(D)}),r_{s}) \subset \Psi(U)$. Choose $t_{0}\in (0,\frac{r_{s}}{2R})$,  where $R=\sup_{z \in (\Psi(\overline{f(D)}))}\|V(X(-s,z))\|$. If $w \in X_{-t}(\Psi(\overline{f(D)}))$ for some  $t \in (s, s+t_{0})$. We have $w=X(-(s+t'), z')$ for some $z'\in \Psi(\overline{f(D)})$ and $0<t'<t_{0}$. Now  we deduce the following: 
	\begin{align*}
		\text{dist}(w,X_{-s}(\Psi(\overline{f(D)))}&\leq \|X(-(s+t'),z')-X(-s,z')\|\\
		&=|t'||V(X(-s,z'))|\\
		&<r_{s}.
	\end{align*}
	It follows that $X_{-t}(\Psi(\overline{f(D)}))\subset \Psi(U)$. Consequently, $\Psi^{-1}(X_{-t}\Psi(\overline{f(D)})) \subset U$ for all $t \in (s, s+t_{0})$. Therefore, $f_{s}$ is a filtering Loewner chain. 
 
 Clearly, for every $t,s \geq 0$ with $0 \leq s \leq t \leq T$  and for any compact subset $K \subset D$, 
	\begin{align*}
		\|f_{s}(z)-f_{t}(z)\|& \leq \sup_{0 \leq \tau \leq T, \xi \in K}\|Df_{t}(\xi)\|\|t-s\|\\
		&\leq \int_{t}^{s}\kappa(\zeta)\,d\zeta,
	\end{align*}
	where $\kappa(\zeta)$ is constant function on $[0,T] \to \mathbb{R}^{+}$. Clearly $\kappa \in L_{loc}^{d}([0,T], \mathbb{R^{+}})$. Let $w \in \cn$. Since $\Psi^{-1}(X_{t}(\Psi(w)) )\to 0$, therefore, there exists  $t>0$ large enough so that $\Psi^{-1}(X_{t}\Psi(w)) \in f(D)$. Choose, $z=f^{-1}(\Psi^{-1}X_{t}\Psi(w)) \in D$. Then, we have $w \in f_{t}(D)$. Therefore, $Rf_{t}(D)=\cn$. Clearly, $Df_{t}(0)=e^{-tDV(0)}$. Therefore, $f_{s}$ is filtering $-DV(0)$ normalized $L^{d}$-Loewner chain.
	\end{proof}
	\begin{lemma}\label{L:RungeC}
	Let $D$ be a bounded strongly convex domain. Then, $\overline{\mathcal{S}_{\mathcal{R}}(D)}= \mathcal{S}_{\mathcal{R}}(D)$, where closure  is taken in compact open topology.  
	\end{lemma}
\begin{proof}
    Let $f \in \overline{\mathcal{S}_{\mathcal{R}}(D)} $. Hence, there exists a sequence $f_{n} \in \mathcal{S}_{\mathcal{R}}(D)$ such that $f_{n}(D)$ is Runge for all $n \in \mathbb{N}$. Now in view of Anders\'en-Lempert theorem \cite[Theorem 2.1]{AnderLemp},  every $f_{n}$ can be approximated by elements of $\text{Aut}(\cn)$ locally uniformly on $D$. Therefore, $f\colon D \to \cn$ can also be approximated by elements  $\text{Aut}(\cn)$ locally uniformly. Since $D$ is Runge, hence  using \cite[Proposition 1.2]{FR1993}, we conclude that $f(D)$ is a Runge domain. Hence, $f \in \mathcal{S}_{\mathcal{R}}(D)  $.
\end{proof}

Here we present proof of \Cref{T:loewner}.

\begin{proof}
	Suppose that $0 \in D \Subset \cn$ is a strongly convex domain  and $f:D \to \cn$ is an univalent map  such that $f \in \mathcal{S}_{\mathfrak{F}}^{1}(D)$. Let $\Om_{s}:=f_{s}(D)$. According to  our assumption, $f$ can
be embedded into a filtering $L^d$-Loewner chain with the Loewner range $\cn$. Hence,  from \cite[Theorem 5.1]{Hamada15},   $(\Om_{s}, Rf_{t}(D))$ is a Runge pair for all $s
	\in [0,\infty)$. Since $Rf_{t}(D)=\cn$, hence,  $\{\Om_{s}\}_{s>0}$ are Runge. Since  $\Om_{s}$ is biholomorphic to $D$,  $\Om_{s}$ are also stein domain for every $s \in [0, \infty)$. Therefore, $\{\Om_{s}\}_{s>0}$ forms a Runge and stein neighborhood basis of $f(D)$. From \Cref{R:steincompacta}, we conclude that $\overline{f(D)}$ is polynomially convex.
	
	 
	 Conversely, suppose that $f\colon D \to \cn$  be a univalent map such that $f(0)=0$ and  $f(D)$ is a bounded  strongly pseudoconvex domain with $\smoo^{m}$ boundary, with $m>2+\frac{1}{2}$  and  $\overline{f(D)}$ is polynomially convex. Now, invoking  \Cref{P:global Narshiman}, we get that there exists $\Psi \in Aut(\cn)$ with $\Psi(0)=0$ such that $\Psi(f(D))$ is convex. Particularly, $\Psi(f(D))$ is spirallike with respect to $-I$. Now in view of \Cref{L:filter}, we conclude that $f \in \mathcal{S}^{1}_{\mathfrak{F}}(D)$.

\end{proof}
\begin{corollary}\label{C:cor1}
	Let  $D$ be a bounded strongly convex domain containing the origin. Then,   $\overline{\mathcal{S}_{\mathfrak{F}}^{1}(D)}=\mathcal{S}_{\mathcal{R}}(D)$ .
\end{corollary}

\begin{proof}
	Clearly, $\mathcal{S}_{\mathfrak{F}}^{1}(D)\subset \mathcal{S}_{\mathcal{R}}(D)$. Thus, we have $\overline{\mathcal{S}_{\mathfrak{F}}^{1}(D)} \subseteq \overline{\mathcal{S}_{\mathcal{R}}(D)}$. From \Cref{L:RungeC}, we conclude that $\overline{\mathcal{S}_{\mathfrak{F}}^{1}(D)} \subseteq \mathcal{S}_{\mathcal{R}}(D)$. Now let $f \in \mathcal{S}_{\mathcal{R}}(D)$. Then by Anders\'en-Lempert theorem \cite[Theorem 2.1]{AnderLemp}, there exist a sequence $\{\Psi_{m}\}_{m \in \mathbb{N}} \in Aut(\cn)$ such that $\psi_{m} \to f$ uniformly over every compact subset of $D$. Since $f(0)=0$ and $df(0)=id$, hence, we can assume that $\psi_{m}(0)=0$ and $d\psi_{m}(0)=id.$ From \cite[Remark 2.5 (ii)]{hamada2020} we conclude that  the convex domain is strictly spirallike with respect to the vector field $-I$. Clearly, $\psi_{m}(D)$ satisfies the condition of \Cref{L:filter}. Therefore,  $\psi_{m}|_{D} \in \mathcal{S}_{\mathfrak{F}}^{1}(D)$. Consequently, $f \in \overline{\mathcal{S}_{\mathfrak{F}}^{1}(D)}$. 
\end{proof}

\section{ Dense holomorphic curves and universal mappings}\label{S:densecurve}
We start this section with a couple of key lemmas that  will be used to prove \Cref{T:densemap1}, \Cref{T:densemap2}.

The following lemma is an application of \Cref{T:Polynostrict}. It will be used  crucially in the proof of \Cref{T:densemap1}.
  \begin{lemma}\label{L:polyapp}
    Let $ \Om \subset \cn $ be a pseudoconvex domain  containing the origin. Assume that $\Om$ is a strictly spirallike domain  with respect to a complete globally asymptotic stable vector field $V \in \mathfrak{X}_{\mathcal{O}}(\cn)$. Let $Y$ be a connected complex manifold. Set $K=(\overline{\mathbb{D}} \cup \{2\}) \times \OOm$ and $I=[1,2] \times \{0\}$. Let $W \subset \mathbb{C}^{n+1}$ be a neighborhood of $K$ and $f \colon W \cup I \to Y$ is a continuous map that is holomorphic on $W$. Then there exists a sequence $\{f_{j} \colon D_{j} \to Y\}_{j \in \mathbb{N}}$ of holomorphic maps from open neighborhood of $K \cup I \subset \mathbb{C}^{n+1}$ such that 
     \begin{enumerate}
         \item [1.]
         $f_{j}\big|_{K \cup I} \to f$ as $j \to \infty$
         \item[2.]
         $f_{j}(2,.)\big|_{\Om}=f(2,.)\big|_{\Omega}$ for all $j \in \mathbb
         {N}$.
     \end{enumerate}
 \end{lemma}
\begin{proof}
Invoking \Cref{T:Polynostrict} we conclude that $\OOm$ is polynomially convex. Now from  \Cref{L:kallin}, we get that $K \cup I$ is polynomially convex. Now invoking  \Cref{R:for04},  we get there exists a sequence of holomorphic maps defined on the neighborhood of $K \cup I$ such that $f_{j}|_{K \cup I} \to f$. Since for every $x \in \Om$ we can construct a hyperplane $L_{x}$ in $\cn$ passing through $x$ and not containing the origin. Consider a closed complex subvarity  $\{2\} \times L_{x}$ and again using \Cref{R:for04}, we get $f_{j}(2,x)=f(2,x)$ for all $x \in \Om$. This proves the lemma.
\end{proof}

\begin{lemma}\label{P: dense1}
    Let $\Om \Subset \cn$ be a bounded pseudoconvex domain containing the 
 the origin which is strictly spirallike with respect to a complete globally asymptotic stable vector field  $V \in \mathfrak{X}_{\mathcal{O}}(\cn)$. Let $Y$ be a connected complex manifold. Let $f \colon \mathbb{D} \to \mco(\Om, Y)$ be a holomorphic map and $\mathcal{U} \subset \mco(\Om, Y)$ a nonempty open subset. Then there exists a sequence of holomorphic maps $\{f_{j}\colon W_{j} \to \mco(\Om, Y)\}_{j \in \mathbb{N}}$ from open neighborhoods $W_{j}$ of $\overline{\mathbb{D}} \cup [1,2] \st \cplx$ such that 
 \begin{enumerate}
     \item [1.]
     $f_{j}\big |_{\mathbb{D}} \to f$
     \item[2.]
     $f_{j}(2) \in \mathcal{U}$ for all $j \in \mathbb{N}$.
 \end{enumerate}
\end{lemma}

\begin{proof}
    Suppose that $\hat{f} \colon \mathbb{D} \times \Om \to Y$ is associated holomorphic map for $f$ defined by $\hat{f}(z,x)=f(z)(x)$. Note that $\mathbb{D}$ and $\Om$ are strictly spirallike with respect to $-id$ and $V$ respectively. Hence, $\mathbb{D} \times \Om$ is strictly spirallike with respect to the vector field $\widetilde{V}(z)=(-z, V(z)) \in \mathfrak{X}_{\mathcal{O}}(\mathbb{C}^{n+1})$. Clearly,  $\widetilde{X}\colon \mathbb{R} \times \mathbb{C}^{n+1}$ defined by  $\widetilde{X}(t,(z,x))=(e^{-t}z, X(t,x))$ is the flow of the vector field $\widetilde{V}$. We have $\widetilde{X}_{t}(\overline{\mathbb{D}} \times \OOm) \subset  \mathbb{D} \times \Om$, for all $t >0$. Hence, for all $t>0$, we get  $\overline{\mathbb{D}} \times \OOm \subset \widetilde{X}_{-t}(\mathbb{D} \times \Om)$.

\noindent    
Define $\widehat{f_{t}} \colon \widetilde{X}_{-t}(\mathbb{D} \times \Om) \to Y$ by $\widehat{f_{t}}(z',x')=\widehat{f}(\widetilde{X}_{t}(z',x'))$. Therefore, $\widehat{f_{t}}$ is defined on a neighborhood of $\overline{\D} \times \OOm$. We now  show that $\widehat{f_{t}} \to \widehat{f}$ locally uniformly on $\mathbb{D} \times \Om$. Let $K \subset \mathbb{D} \times \Om$ be a compact set. Since $\widetilde{X}_{t} \to id_{\D \times \Om}$ as $t \to 0^{+}$ uniformly on $K$ and $\widehat{f}$ is a uniformly continuous  map on every compact subset, hence $\widehat{f_{t}} \to \widehat{f}$  uniformly on $K$ as $t \to 0^{+}$. Let $ u \in \mathcal{U}$. Similarly, for all $t>0$ we  consider $\hat{u}_{t}: X_{t}(\Om) \to Y$ defined by $\hat{u}_{t}(x)=u(X(t,x))$. Then, for each $t>0$, we have $u_{t} \in \mco(\OOm, Y)$ and $u_{t} \to u$ locally uniformly on $\Om$. Since $\hat{f}$ can be approximated locally uniformly by holomorphic functions  defined on $\overline{\D} \times \OOm$, it is enough to consider that $\hat{f} \in \mco(\overline{\D}\times \OOm)$. Similarly, we can assume that $\hat{u} \in \mco(\OOm, Y)$.


\noindent
 Now  proceeding in a similar way as \cite[Proof of Lemma 2.1]{kus2017} and using \Cref{L:polyapp}, we conclude that there exist  sequence $\{g_{j} \colon V_{j} \to Y\}_{j \in \mathbb{N}}$ of holomorphic map from open neighborhoods $V_{j}$ of $\big((\overline{D} \cup \{2\})\times \OOm\big) \cup ([1,2] \times \{0\}) \subset \mathbb{C}^{n+1}$ such that 
\begin{enumerate}
    \item [i.]
    $g_{j} \big|_{\D \times \Om} \to \hat{f}$ as $j \to \infty$ locally uniformly
    \item [ii.]
    $g_{j}(2, .)\big|{\Om} =u$, for all $j \in \mathbb{N}$.
\end{enumerate}


 Since $\Om$ is  a strictly spirallike domain with respect to  $V$, hence for any open neighbourhood $N_{\OOm}$ of $\OOm$, we have $X_{t}(\OOm) \subset N_{\OOm}$, for all $t \in [0,1]$. From the continuity of the map $X_{t}$ and compactness of $[0,1] \times \OOm$, we obtain an open set  $U\times G \st \cplx \times \mathbb{C}^{n+1}$, with  $[0,1] \times \OOm \st U \times G$ and $X_{t}(G) \st N_{\OOm}$ for all $t \in U$. Now again proceeding  the same way as  \cite[Proof of Lemma 2.1]{kus2017} and using \Cref{L:polyapp}, we conclude the lemma.


    \end{proof}
 The next lemma is an application of \cite[Lemma 1]{FW2005}. It can be proved similarly as \cite[lemma 2.2]{kus2017}, using \Cref{P: dense1}. Therefore, we omit its proof.
 \begin{lemma}\label{L:l1}
     Let $\{D_{j}\}_{j \in \mathbb{N}}$ be  a sequence of open neighborhood of $\D \cup [1,2] \st \cplx$. Then there exists a sequence $\{\phi_{j} \colon \D \to D_{j}\}_{j \in \mathbb{N}}$ of holomorphic maps such that for all $j \in \mathbb{N}$ the following holds
      \begin{itemize}
         \item [1.]
        $\phi_{j}\big|_{\D} \to id\big|_{\D}$
        \item[2.]
       $ 2 \in \phi_{j}(\D)$
     \end{itemize}
 \end{lemma}

\begin{definition}
    Let $X, Y$ be topological spaces. A sequence $\{f_{\nu}\}  \st \smoo(X, Y)$ is compactly diverges if for every pair of compacts $ H \subseteq X$ and $K \subseteq Y$ there exists $\nu_{0} \in \mathbb{N}$ such that $f_{\nu}(H) \cap  K= \emptyset$ for every $\nu \geq \nu_{0}$. 
    \end{definition}
 \noindent
 A  domain $\Om \St \cn$  is said to be taut if every sequence $\{f_{n}\}_{n} \st \mco(\D, \Om)$  is compactly divergent in $\mco(\D, \Om)$ or has a subsequence convergent in $\mco(\D, \Om)$. If $\Om$ is taut, then, from \cite[Theorem 5.1.5]{Kobaya},  we get that for every complex manifold $Y$ and every sequence sequence $\{f_{n}\}_{n \in \mathbb{N}} \st  \mco(Y, \Om)$ either $\{f_{n}\}$ has a convergent subsequence or $\{f_{n}\}$ is compactly divergent.

 Next lemma will be used to prove \Cref{T:densemap2}.
 \begin{lemma}\label{L:GT2}
     Let $\widetilde{\Om} \St \cn $ be $c_{\widetilde{\Om}}$ finitely compact   pseudoconvex domain.  Assume that  $\Om \St \cn  $ contains the origin is taut domain and strictly spirallike with respect to complete globally asymptotic stable vector field $V \in \mathfrak{X}_{\mathcal{O}}(\mathbb{\cplx}^{n})$. Let $\tau \in Aut(\widetilde{\Om})$ be an automorphism such that $\{\tau^{j}\}_{j \in \mathbb{N}}$ is compactly divergent. Then the map $C_{\tau}\colon\mco(\widetilde{\Om}, \Om) \to \mco(\widetilde{\Om}, \Om)$  defined by $C_{\tau}(f)=f\circ \tau $ is hypercyclic with respect to the sequence $(j)_{j \in \mathbb{N}}$.
 \end{lemma}
\begin{proof}
 We have to show that there exists $F \in \mco(\widetilde{\Om}, \Om)$ such that $\{C_{\tau}^{j}(F):j \in \mathbb{N} \}$ is dense in $\mco(\widetilde{\Om}, \Om)$  with respect to compact open topology.  In view of \Cref{R:brik1}, it is enough to show that for any pair of nonempty open subsets  $G, U\subset \mco(X, \Om)$, there exists $j \in \mathbb{N}$ such that $(C_{\tau})^{j}(G) \cap U \neq 0$. Let $g \in G$ and $h \in U$. Since $\Om$ is strictly spirallike domain with respect to the vector filed $V$, hence, for any $t>0$, we obtain $\overline{X_{t}(g(\widetilde{\Om}))} \subseteq X_{t}\big(\overline{g(\widetilde{\Om})}\big) \subseteq X_{t}(\OOm)$. Clearly, $X_{t}(\OOm) \subset \Om$ is a compact set, for all $t>0$ . Hence, $X_{t} \circ g (\widetilde{\Om}) \St \Om$. Here, $G$ is an open set containing $g$. Since $X_{t} \circ g \to g$ as $t \to 0^{+}$ in compact open topology,  hence,  for small enough $t>0$, we conclude that $g_{t}:=X_{t} \circ g$ in $G$. Therefore, without loss of generality we can assume that $g(\widetilde{\Om}), h(\tilde{\Om}) \St \Om$.
For $j \in \{1,2\}$ consider the following maps $\pi_{j}\colon \Om \times \Om  \to \Om$ defined by $\pi_{j}(x_{1},x_{2})=x_{j}$. Clearly,  $\Om \times \Om$ is spirallike domain with respect to the vector field $(V,V) \in \mathfrak{X}_{\mathcal{O}}(\cplx^{2n})$. Hence, from \Cref{Re:r2}, there exists $x_{1}, x_{2} \in \D$ and a holomorphic map $f \colon \D \to \mco(\Om \times \Om, \Om)$ such that $f(x_{1})=\pi_{1}, f(x_{2})=\pi_{2}$.  Hence, the rest of the proof goes the same as \cite[Lemma 4.1]{kus2017}.
\end{proof}

Now we present the proof of \Cref{T:densemap1}.

\begin{proof}[Proof of \Cref{T:densemap1}]
The topological space  $\mco(\mathcal{D}, Y)$ has a countable base with respect to the compact open topology (see \cite[Remark 1.1]{kus2017}). Hence, we can   choose a countable base $\{\mathcal{U}_{j}\}_{j \in \mathbb{N}}$ for the topological space $\mco(\mathcal{D}, Y)$ such that $\mathcal{U}_{j} \neq \emptyset$. Consider the set $\mathcal{W}_{j}=\{f \in \mco(\D ,\mco(\mathcal{D}, Y)): f(\D) \cap \mathcal{U}_{j} \neq \emptyset\}$. Clearly, the set is open. We will show that $\mathcal{W}_{j}$ is dense subsets of $\mco(\D, \mco(\mathcal{D}, Y))$ with respect to compact open topology for all $j \in \mathbb{N}$. Since $\mco(\D \times \mathcal{D}, Y)$ is a Baire space, hence, $\mco(\D,\mco(\mathcal{D}, Y))$ is also a Baire space. Therefore,   countable  intersection of open dense subsets in $\mco(\D, \mco(\mathcal{D}, Y))$ is again dense. Consequently, we conclude that $\mathcal{W}:=\cap_{j \in \mathbb{N}}\mathcal{W}_{j}$ is dense in $\mco(\D, \mco(\mathcal{D}, Y))$. 

Let $g \in \cap_{j \in \mathbb{N}}\mathcal{W}_{j}$ and $\mathcal{V}$ be any open subset of $\mco(\mathcal{D}, Y)$. Then, there exists $\mathcal{U}_{j}$ such that $\mathcal{U}_{j} \subset \mathcal{V}$. From the choice of $g$, there exists $z_{j} \in \D$ such that $g(z_{j}) \in \mathcal{U}_{j}$. Hence, for any open subset $\mathcal{V} \st \mco(\mathcal{D}, Y)$ we have $g(\D) \cap \mathcal{V} \neq \emptyset$. Therefore, $\overline{g(\D)}=\mco(\mathcal{D}, Y)$. Consequently, we get that $\mco(\mathcal{D}, Y)$ is $\D$-dominated. 

Now we show that  $\mathcal{W}_{j}$ are dense in $\mco(\D, \mco(\mathcal{D}, Y))$. Fix $ k \in \mathbb{N}$. Let $f \in \mco(\D, \mco(\mathcal{D}, Y))$. Now, invoking \Cref{P: dense1}, we get a sequence  $\{f_{m}\colon {\mathcal{D}}_{m} \to \mco(\mathcal{D} , Y)\}$ of holomorphic map from open neighbourhood of $\overline{\D} \cup [1,2]$ such that $f_{m}\big|_{\D} \to f$ as $m \to \infty$ and $f_{m}(2) \in \mathcal{U}_{k}$ for all $m \in \mathbb{N}$. Now, using \Cref{L:l1}, there exists a sequence of holomorphic maps $\phi_{m} \colon \D \to \mathcal{D}_{m}$ such that $\phi_{m} \to id$ locally uniformly on $\D$ and there exists $z_{m} \in \D$ such that $2 = \phi_{m}(z_{m})$ for all $m \in \mathbb{N}$. Now, we consider the sequence of the holomorphic  map $f_{m}\circ  \phi_{m} \colon \D \to \mco(\mathcal{D}, Y)$. Since $\phi_{m} \to id$ locally uniformly, hence,  we have  $f_{m}\circ  \phi_{m}\big|_{\D} \to f$ locally uniformly. Now, for all $m \in \mathbb{N}$, we have $f_{m}(\phi_{m}(z_{m}))=f_{m}(2) \in \mathcal{U}_{k}$.  Therefore, we conclude that for any $k \in \mathbb{N}$ and $f \in \mco(\D, \mco(\mathcal{D},Y))$, there exists  a sequence $ g_{m}:=f_{m}\circ \phi_{m} \in \mathcal{W}_{k}$ such that $g_{m} \to f$ locally uniformly $\D$. Consequently, each $\mathcal{W}_{k}$ are dense.

\end{proof}

\begin{remark}\label{Re:r1}
From the above proof we obtain that $\overline{\mathcal{W}}=\mco(\D,\mco(\mathcal{D}, Y))$ and for all $f \in \mathcal{W}$ we have $\overline{f(\D)}=\mco(\mathcal{D}, Y)$. Therefore, the set of all dense holomorphic maps $f\colon \D \to \mco(\mathcal{D}, Y)$ is dense in $\mco(\D,\mco(\mathcal{D}, Y))$.
\end{remark}
 \begin{remark}\label{Re:r2}
With a little modification  of the proof of the \Cref{P: dense1} and \Cref{L:l1} we can prove that if $u_{1}, u_{2} \in\mco(\Om, Y)$ such that both $u_{1}, u_{2}$ has holomorphic extension on $\OOm$, then, there exists a holomorphic map $f\colon \D \to \mco(\Om, Y)$ and $x_{1}, x_{2} \in \D$  and such that $f(x_{1})=u$ and $f(x_{2})=v$.

 \end{remark}

Now we present the proof of \Cref{T:densemap2}.
\begin{proof}[Proof of \Cref{T:densemap2}]
Given that  $\mathcal{D} \St \cn$ is  a bounded  pseudoconvex domain  containing the origin that is strictly spirallike with respect to the complete globally asymptotic stable vector field $V \in \mathfrak{X}_{\mathcal{O}}(\mathbb{\cplx}^{n})$  and $c_{\mathcal{D}}$-finitely compact. We have to show that there exists a holomorphic map  $G\colon \mathcal{D}\to Y$ such that the set  $\{G\circ \tau^{j}|j \in \mathbb{N}\}$ is dense in $\mco(\mathcal{D}, Y)$ with respect to the compact open topology. From \Cref{T:densemap1}, there exists  a holomorphic map $f \colon \D \to \mco(\mathcal{D}, Y)$ such  that $\overline{f(\D)}=\mco(\mathcal{D}, Y)$. Let  $\widehat{f} \colon \D \times \mathcal{D} \to Y$ defined by $\widehat{f}(z,x)=(f(z))(x)$ is  associated holomorphic map  induced by $f$. Clearly, we obtain a continuous map $\widehat{f}_{*}\colon \mco(\mathcal{D}, \D \times \mathcal{D}) \to \mco(\mathcal{D}, Y)$ defined by $\widehat{f}_{*}(g)=\widehat{f} \circ g$. At first  we show that $\overline{\widehat{f}_{*}(\mco(\mathcal{D}, \D \times \mathcal{D}) )}=\mco(\mathcal{D}, Y)$. Let $\{U_{m}\}_{m \in \mathbb{N}}$ be countable basis for $\mco(\mathcal{D}, Y)$. Since $f$ is dense holomorphic map, hence, there exists $z_{m} \in \D$ such that $f(z_{m}) \in U_{m}$. Now define $g_{m} \colon \mathcal{D} \to \D \times \mathcal{D}$ by $g_{m}(x)=(z_{m},x)$. Then, we have $\widehat{f}_{*}(g_{m}(x))=\widehat{f}(z_{m},x)=f(z_{m})$. Since $\{U_{m}\}$ are basic open sets, hence, we conclude that for every open set $U \in \mco(\mathcal{D}, Y)$ there exists $g \in \mco(\mathcal{D}, \mco(\D, \mathcal{D}))$ such that $\widehat{f}(g) \in U $.

Since $\mathcal{D}$ is $c_{\mathcal{D}}$-finitely compact hence it is particularly a taut domain. Therefore, $\D \times \mathcal{D}$ is a taut domain and also   strictly spirallike with respect to complete globally asymptotic stable vector field $(-I, V) \in \mathfrak{X}_{\mathcal{O}}(\mathbb{\cplx}^{1+n})$. Since $\mathcal{D}$ is a bounded pseudoconvex domain that is $c_{\mathcal{D}}$-finitely compact and $\{\tau^{j}\}$ compactly diverges on $\mathcal{D}$, hence, invoking \Cref{L:GT2}, we infer that there exists  $\mco(\mathcal{D} ,\D \times \mathcal{D}) $-universal map $F$   for $\tau$. Hence, we obtain that $\overline{\{F\circ \tau^{j}| j \in \mathbb{N}\}}=\mco(\mathcal{D}, \D \times \mathcal{D}) $. Let  us consider the map $G=\widehat{f}_{*}(F) \in \mco(\mathcal{D}, Y)$.  We show that $\{G\circ \tau^{j}| j \in \mathbb{N}\}$ is dense in $\mco(\mathcal{D}, Y)$. 

 Let  $\mathcal{U} \st \mco(\mathcal{D}, Y)$  be any open set. Since $\overline{\widehat{f}_{*}(\mco(\mathcal{D}, \D \times\mathcal{D}) )}=\mco(\mathcal{D}, Y)$, hence, there exists $g \in \mco(\mathcal{D}, \D \times \mathcal{D})$ such that $\widehat{f}_{*}(g) \in \mathcal{U}$.
Now $\widehat{f}_{*}$ is a  continuous map. Therefore, $\hat{f}_{*}^{-1}(\mathcal{U})$ is an open set in $\mco(\mathcal{D}, \D \times \mathcal{D})$ containing $g$. Since, $\overline{\{F \circ \tau^{j}|j \in \mathbb{N}\}}=\mco(\mathcal{D}, \D \times \mathcal{D})$, hence, there exists $j_{0} \in \mathbb{N}$ such that $F\circ \tau ^{j_{0}} \in \widehat{f}_{*}^{-1}(\mathcal{U})$. Consequently,  $\hat{f}_{*}\circ F \circ \tau^{j_{0}} \in \mathcal{U}$. Since we have $G\circ \tau^{j}=\hat{f} \circ F \circ \tau^{j}$, therefore, we conclude  that for every open subset $\mathcal{U}$ of $\mco(\mathcal{D} ,Y)$ there exists $j_{0} \in \mathbb{N}$ such that $G \circ \tau^{j_{0}} \in \mathcal{U}$. This proves the existence of $\mco(\mathcal{D}, Y)$-universal map for $\tau$. This proves the theorem.
\end{proof}
\begin{corollary}
 Let $\Om \St \cn$ as \Cref{T:densemap2} and $\tau \in Aut(\Om)$ is generalized translation. If there exists $\mco(\Om,Y)$-universal map for $\tau$, then  $\mco(\Om ,Y)$ is $\D$-dominated.     
 \end{corollary}
This can be seen as follow:  Let $F\colon \mathcal{D} \to Y$ be a $\mco(\mathcal{D}, Y)$-universal map for $\tau$. From \Cref{T:densemap1}, we get a holomorphic map $f\colon \D \to \mco(\mathcal{D}, \mathcal{D})$ such that image of $f$ is dense in $\mco(\mathcal{D}, \mathcal{D})$.  Now $F_{*}\colon \mco(\mathcal{D}, \mathcal{D}) \to \mco(\mathcal{D}, Y)$ defines a  continuous map by $F_{*}(g)=F\circ g$. Then $F_{*} \circ f \colon \D \to \mco(\mathcal{D}, Y)$ is a dense holomorphic map.

\begin{proof}[Proof of \Cref{T:densemap3}]
    $(1) \impl (2)$ follows from definition. Since $\mathcal{D}$ is spirallike with respect to complete globally asymptotic stable vector field $V \in \mathfrak{X}_{\mathcal{O}}(\mathbb{\cplx}^{n})$, hence it  is contractible (see\cite[Page 21]{CG}). Now, from  \cite[Theorem 1]{Hung94}, we get $(2)\Leftrightarrow (3)$. Suppose that  $(3)$ holds. Since $\mathcal{D}$ is $c_{\mathcal{D}}$-finitely compact, hence, from \cite[Lemma 10]{Zajac16}, we obtain that $C_{\tau} \colon \mco(\Om) \to \mco(\Om)$ is hypercyclic with respect to $(n)$. Then, from \cite[corollary 8]{Zajac16} and \cite[Lemma 2]{Zajac16} we conclude  that, for every compact $\mathcal{O}(\mathcal{D})$-convex subset $K$, there exist $j_{K}$ such that $K \cup \tau^{j_{K}}(K)$ is $\mathcal{O}(\mathcal{D})$-convex. Hence, $(3) \impl (1)$. 
    
   
\end{proof}

\section{ examples}\label{S:Example}

\begin{example}\label{E:ex2}
    Let us consider the following domain: Let $r>0$ and 
  $$
  \Omega=\{(z_1,z_2)\in \mathbb{C}^{2}\;|\;|z_{1}|<r,~|z_2|<e^{-|z_{1}|})\}.
  $$
  Clearly, $\Om$ is a Hartogs domain (see \cite[Page 406]{Stoutbook}) over an open ball of radius $r$ in $\cplx$. Since $-\log{e^{-|z_{1}|}}=|z_{1}|$ is plurisubharmonic function hence from \cite[page 406]{Stoutbook}, we get that $\Om$ is pseudoconvex domain.
  
  Now consider the vector field $F(z_1,z_2)=(-2z_1,-3z_2+z_1z_2)$. The flow of the vector field is defined by $X(t,z)= \left(z_1e^{-2t},z_2e^{-3t}e^{\frac{z_1}{2}(1-e^{-2t})}\right)$, where $(t,z) \in \mathbb{R} \times \mathbb{C}^{2}$. We show that $\Om$ is strictly spirallike with respect to $F$.
  Let $(z_{1},z_{2}) \in 
  \Om$ and $t>0$. Suppose that $w_{1}=z_{1}e^{-2t}$ and $w_{2}=z_2e^{-3t}e^{\frac{z_1}{2}(1-e^{-2t})}$. Clearly, $|w_{1}|<r$ for all $t>0$. Since, for all $t>0$, we have $(1-e^{-2t})(-|z_{1}|+\frac{\rl(z_{1})}{2})-3t<0$, hence for all $t>0$, we have the following 
  \begin{align*}
  |w_{2}|&=|z_{2}|e^{-3t+\frac{\rl(z_{1})}{2}(1-e^{-2t})}\\
  &\leq e^{-|z_{1}|-3t+\frac{\rl(z_{1})}{2}(1-e^{-2t})}\\
  &<e^{-|w_{1}|}.
\end{align*}
Therefore, $\Om$ is a bounded pseudoconvex domain that is a strictly spirallike domain with respect to the complete globally asymptotic stable vector field $F$. Clearly, $\Om$ is not  convex.

Let for $j =\{1,2\}$ $V_{j}:=\{(z_{1},z_{2}) \in \CC| z_{j} =0\}$. Clearly, $\Om$ is a pseudoconvex Reinhardt domain such that $\Om \cap V_{j} \neq \emptyset$. Therefore, from \cite[Theorem 2]{Zwo2000} it follows that $\Om$ is $c_{\Om}$-finitely compact.

Therefore, the conclusion of \Cref{T:densemap1}, \Cref{T:densemap2} is  true for this domain. 
\end{example}

Next, we give an  example of a  non-convex, strongly pseudoconvex domain with polynomially convex closure that is biholomorphic to a bounded strongly convex domain. Therefore, the conclusion of  \Cref{P:global Narshiman} as well as \Cref{T:loewner} holds.
\begin{example} \label{E:ex1}
	Let $\Om_{1}=\{(z_{1},z_{2})\in \mathbb{C}^{2}:|z_{1}|^2+|z_2|^{2}+|z_{1}|^{2}|z_{2}|^2-1<0\}$.  Here, $\Om_{1}$ is a strongly convex domain. Since $\Om_{1}$ is a circular domain hence from Cartan's Theorem \cite{cartan}, it follows that any biholomorphism from the open unit ball  onto $\Om_{1}$ is a linear map.  Therefore, the defining function of $\Om_{1}$ can not contain the term $|z_{1}|^{2}|z_{2}|^{2}$. Hence, $\Om_{1}$ is not biholomorphic to the open unit ball. 
 
 Let $U \st \cplx$ be any non-convex simply connected domain.  Suppose that $p,q \in U$ can not be connected by a straight line contained in $U$. By Riemann mapping theorem, we get  there exists a biholomorphism  $f\colon \D \to U$ such that $f(0)=p$ and $f(x)=q$. Clearly $x \neq 0$. Choose $\eps>0$ such that $0<\eps <\frac{1}{|x|}-1$. Let $\D(0, 1+\epsilon)=\{z \in \cplx| |z|<1+\eps\}$. Note that $\overline{\Om}_{1} \St \D(0, 1+\epsilon) \times \D(0, 1+\epsilon)$.

 Let  $G\colon \D(0, 1+\epsilon) \times \D(0, 1+\epsilon) \to U \times \D(0, 1+\epsilon)$ defined by  $G(z,w)=(f(\frac{z}{1+\eps}),w)$. Here $G(0,0)$ and $G((1+\eps)x,0)$ can not be connected by a straight line by construction. Hence, $G(\Om_{1})$ is not convex. Since $G$ has holomorphic extension on a neighborhood of $\overline{\Om}_{1}$, hence $G(\Om_{1})$ is strongly pseudoconvex domain with $\smoo^{\infty}$ boundary. Note that $G$ is a biholomorphism from a star-shaped domain onto a Runge domain. Hence, in view of \cite[Theorem 2.1]{AnderLemp}, we conclude that $G$ can be approximated by $Aut(\CC)$. Consequently, $G(\overline{\Om}_{1})$ is polynomially convex.
 Therefore, from \Cref{P:global Narshiman}, we conclude that there exists  $\Psi \in \text{Aut}(\cn)$ such that $\Psi(G(\Om_{1}))$ is convex. Equivalently, $G$ can be embedded into a filtering Loewner chain.
\end{example}

 \noindent {\bf Acknowledgements.} 
Sanjoy Chatterjee is supported by a CSIR fellowship (File No-09/921(0283)/2019-EMR-I) and also would like to thank Golam Mostafa Mondal for several discussions and fruitful comments. Sushil Gorai is partially supported by a Core Research Grant (CRG/2022/003560) of SERB, Government of India.

\bibliographystyle{plain}
\bibliography{biblio}

\end{document}